%% file: honeycomb_stability6.tex
\documentclass[reqno,11pt]{amsart}

\usepackage{amsmath,amssymb,amsthm,mathrsfs}
\usepackage{epsfig,color}
\usepackage[latin1]{inputenc}

\usepackage[numbers]{natbib}

\numberwithin{equation}{section}

\def\tt{\mathbf{t}}

\def\H{\mathcal H}

\def\R{\mathbb R}
\def\N{\mathbb N}
\def\Z{\mathbb Z}

\def\arc{{\rm arc}}

\newcommand{\dist}{\mathop{\mathrm{dist}}}
\def\e{\varepsilon}
\def\s{\sigma}
\def\S{\Sigma}

\def\Div{{\rm div}\,}
\def\om{\omega}
\def\l{\lambda}
\def\g{\gamma}
\def\k{\kappa}

\def\de{\delta}
\def\Id{{\rm Id}}

\def\spt{{\rm spt}}

\def\pa{\partial}

\def\00{{\bf 0}}

\def\bd{{\rm bd}\,}
\def\INT{{\rm int}\,}

\def\E{\mathcal{E}}

\def\F{\mathcal{F}}

\def\T{\mathcal{T}}

\newcommand{\vol}{\mathrm{vol}\,}

\renewcommand{\a}{\alpha}
\renewcommand{\b}{\beta}
\renewcommand{\d}{\mathrm{d}}
\newcommand{\hd}{\mathrm{hd}}

\renewcommand{\l}{\lambda}

\renewcommand{\om}{\omega}

\renewcommand{\Div}{{\rm div \,}}
\newcommand{\ov}{\overline}

\newcommand{\diam}{\mathrm{diam}}
\newcommand{\cc}{\subset\subset}

\def\Lip{{\rm Lip}\,}

\def\PHI{\mathbf{\Phi}}
\def\PSI{\mathbf{\Psi}}

\def\ttau{\boldsymbol{\tau}}

\def\tt{\mathbf{t}}

\newtheorem*{theorem*}{Theorem}
\newtheorem{theorem}{Theorem}[section]
\newtheorem{lemma}[theorem]{Lemma}
\newtheorem{proposition}[theorem]{Proposition}

\newtheorem{corollary}[theorem]{Corollary}
\newtheorem{remark}[theorem]{Remark}

\title{A sharp quantitative version of Hales' isoperimetric honeycomb theorem}

\author{M. Caroccia}
\address{Dipartmento di Matematica, Universit\`a di Pisa, Largo Bruno Pontecorvo 5, 56127 Pisa, Italy}
\email{caroccia.marco@gmail.com}

\author{F. Maggi}
\address{Department of Mathematics, University of Texas at Austin, Austin, TX, USA}
\email{maggi@math.utexas.edu}

\begin{document}

\begin{abstract} We prove a sharp quantitative version of Hales' isoperimetric honeycomb theorem by exploiting a quantitative isoperimetric inequality for polygons and an improved convergence theorem for planar bubble clusters. Further applications include the description of isoperimetric tilings of the torus with respect to almost unit-area constraints or with respect to almost flat Riemannian metrics.
\end{abstract}

\maketitle

\section{Introduction} The isoperimetric nature of the planar ``honeycomb tiling''  has been apparent since antiquity. Referring to \cite[Section 15.1]{Morgan} for a brief historical account on this problem, we just recall here that Hales' isoperimetric theorem, see inequality \eqref{hexagonal honeycomb thm torus} below, gives a precise formulation of this intuitive idea. Our goal here is to strengthen Hales' theorem into a quantitative statement, similarly to what has been done with other isoperimetric theorems in recent years (see, for example, \cite{fuscomaggipratelli,FigalliMaggiPratelliINVENTIONES}).

Following \cite[Chapters 29-30]{maggiBOOK}, we work in the framework of sets of finite perimeter. A {\it $N$-tiling} $\E$ of a two-dimensional torus $\T$ is a family $\E=\{\E(h)\}_{h=1}^N$ of sets of finite perimeter in $\T$ such that $|\T\setminus\bigcup_{h=1}^N\E(h)|=0$ and $|\E(h)\cap\E(k)|=0$ for every $h,k\in\N$, $h\ne k$.  The volume of $\E$ is $\vol(\E)=(|\E(1)|,...,|\E(N)|)$, and the relative perimeter of $\E$ in $A\subset\T$ is given by
\[
P(\E;A)=\frac12\sum_{h=1}^N P(\E(h);A)\,,
\]
(where $P(E;A)=\H^1(A\cap\pa E)$ if $E$ is an open set with Lipschitz boundary), while the distance between two tilings $\E$ and $\F$ is defined as
\[
\d(\E,\F)=\frac12\sum_{h=1}^N|\E(h)\Delta\F(h)|\,.
\]
We say that $\E$ is a {\it unit-area tiling} of $\T$ if $|\E(h)|=1$ for every $h=1,...,N$. (In particular, in that case, it must be $N=|\T|$). Let $\hat H$ denote the reference unit-area hexagon in $\R^2$ depicted in
\begin{figure}
  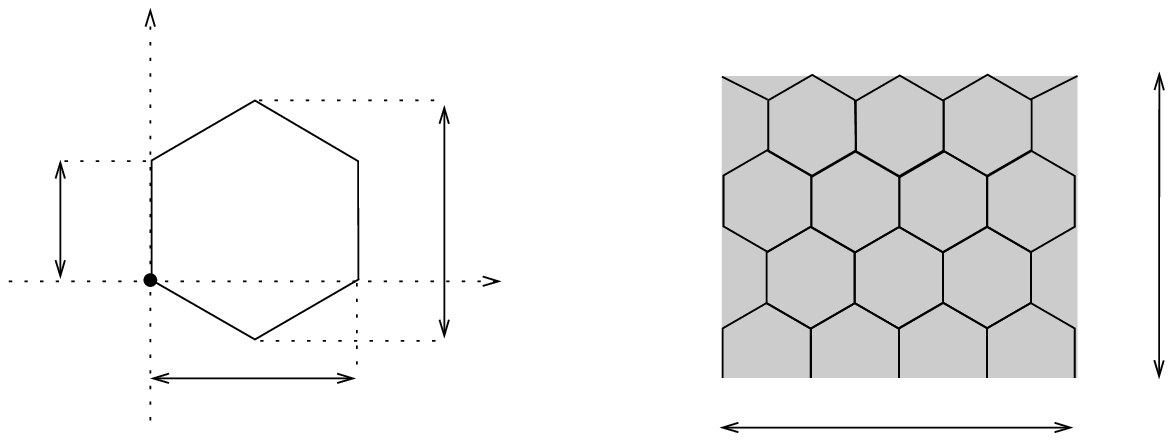\caption{{\small Thoroughout the paper $\hat H$ denotes the unit-area regular hexagon in $\R^2$ depicted on the left and we set $H=\hat H/_{\approx}$. Since $|H|=1$, one has $P(H)=2(12)^{1/4}$, and the side-length of $H$ is thus $\ell=(12)^{1/4}/3$. On the right, the torus $\T$ (depicted in gray) and the reference unit-area tiling $\H$ of $\T$ (with $\a=\b=4$). Notice that $N=|\T|=\a\,\b$. The chambers of $\H$ are enumerated so that $\H(1)=H$, $\{\H(h)\}_{h=1}^\b$ is the bottom row of hexagons in $\T$, and, more generally, if $0\le k\le\a-1$, then $\{\H(h)\}_{h=1+k\beta}^{(k+1)\beta}$ is the $(k+1)$th row of hexagons in $\T$.}}\label{fig hexagon}
\end{figure}
Figure \ref{fig hexagon}, so that $\ell=(12)^{1/4}/3$ is the side-length of $\hat H$. Given $\a,\b\in\N$, let us consider the torus $\T=\T_{\a,\b}=\R^2/_\approx$ where
\[
(x_1,x_2)\approx(y_1,y_2)\qquad\mbox{if and only if}\qquad \mbox{$\exists h,k\in\N$ s.t.}\quad
\left\{\begin{array}{l}
x_1=y_1+h\,\beta\,\sqrt{3}\,\ell\,,
\\
x_2=y_2+k\,\a\,\frac{3}2\,\ell\,,
\end{array}\right .
\]
and set $H=\hat H/_{\approx}\subset\T$. In order to avoid degenerate situations, {\it we shall always assume that}
\begin{equation}
  \label{occhio}
  \mbox{$\a$ is even and $\beta\ge 2$}\,.
\end{equation}
 In this way, $H$ is a regular unit-area hexagon (i.e., the vertexes of $\hat H$ belong to six different equivalence classes) and one obtains a reference unit-area tiling $\H=\{\H(h)\}_{h=1}^N$ of $\T$ consisting of $\a$ rows and $\beta$ columns of regular hexagons by considering translations of $H$ by $(h\,\sqrt{3}\ell,3\ell\,k/2)$ ($h,k\in\Z$); see again Figure \ref{fig hexagon}. Under this assumption, {\it Hales' isoperimetric honeycomb theorem} asserts that
\begin{equation}
  \label{hexagonal honeycomb thm torus}
  P(\E)\ge P(\H)\,,
\end{equation}
whenever $\E$ is a unit-area tiling of $\T$, and that $P(\E)=P(\H)$ if and only if (up to a relabeling of the chambers of $\E$) one has $\E(h)=v+\H(h)$ for every $h=1,...,N$ and for some $v=(t\sqrt{3}\ell,s\ell)$ with $s,t\in[0,1]$. Our first main result strengthens this isoperimetric theorem in a sharp quantitative way.

\begin{theorem}
  \label{thm main periodic}
  There exists a positive constant $\k$ depending on $\T$ such that
  \begin{equation}
  \label{hexagonal honeycomb thm torus quantitative}
  P(\E)\ge P(\H)\,\Big\{1+\k\,\a(\E)^2\Big\}\,,
  \end{equation}
  whenever $\E$ is a unit-area tiling of $\T$ and
  \[
  \a(\E)=\inf \d(\hat{\E},v+\H)
  \]
  where the minimization takes place among all $v=(t\sqrt{3}\ell,s\ell)$, $s,t\in[0,1]$, and among all tilings $\hat\E$ obtained by setting $\hat\E(h)=\E(\s(h))$ for a permutation $\s$ of $\{1,...,N\}$. (Recall that the chambers of the reference honeycomb $\H$ are enumerated in a specific way, see Figure \ref{fig hexagon}.)
\end{theorem}

\begin{remark}
  {\rm We notice that \eqref{hexagonal honeycomb thm torus quantitative} is sharp in the decay rate of $\a(\E)$ in terms of $P(\E)-P(\H)$. Indeed, if $\om:(0,\infty)\to(0,\infty)$ is such that $P(\E)\ge P(\H)(1+\om(\a(\E)))$ for every unit-area tiling $\E$, then, for some $s_0>0$, one must have $\om(s)\le C\,s^2$ for $s\in(0,s_0)$. Indeed, one can explicitly construct a one-parameter family $\{\E_t\}_{0<t<\e}$ of unit-area tilings of $\T$ such that $P(\E_t)\le P(\H)(1+C\,\a(\E_t)^2)$ and $\{\a(\E_t):t\in(0,\e)\}=(0,s_0)$, so that $\om(s)\le C\,s^2$ for every $s\in(0,s_0)$.}
\end{remark}

In Theorem \ref{hexagonal honeycomb thm torus quantitative small} below, inequality \eqref{hexagonal honeycomb thm torus quantitative} is proven in much stronger form for $\pa\E$ in a special class of $C^1$-small $C^{1,1}$-diffeomorphic images of $\pa\H$, see \eqref{stabilita hd} and \eqref{stabilita f C1}.  The two main ingredients in the proof of Theorem \ref{hexagonal honeycomb thm torus quantitative small} are: a quantitative version of the hexagonal isoperimetric inequality, which we deduce from \cite{shilleto,indreinurbekyan}, see Lemma \ref{lemma indrei}; and a quantitative version of Hales' hexagonal isoperimetric inequality (the key tool behind Hales' proof of \eqref{hexagonal honeycomb thm torus}), proved in Lemma \ref{lemma hales}. These inequalities allow one to prove that each chamber of the unit-area tiling $\E$ is actually close, in terms of the size of $P(\E)-P(\H)$, to some regular unit-area hexagon in $\T$. These hexagons have no reason to fit nicely into an hexagonal honeycomb of $\T$ (that is, a translation of $\H$), therefore we need an additional argument to show that, up to translations and rotations of order $P(\E)-P(\H)$, one can achieve this. Having completed the proof of Theorem \ref{hexagonal honeycomb thm torus quantitative small}, we deduce Theorem \ref{thm main periodic} by a contradiction argument based on an improved convergence theorem for planar bubble clusters that was recently established in \cite{CiLeMaIC1}, and along the lines of the selection principle method proposed in \cite{CicaleseLeonardi}. Another consequence of Theorem \ref{hexagonal honeycomb thm torus quantitative small}, obtained in a similar vein, is the following result, which gives a precise description of isoperimetric tilings of $\T$ subject to an ``almost unit-area'' constraint.

\begin{theorem}
  \label{thm pertub volumes}
  There exist positive constants $C_0,\de_0$ depending on $\T$ with the following property. If $\sum_{h=1}^Nm_h=N$ with $m_h>0$ and $|m_h-1|<\de_0$ for every $h=1,...,N$, and if $\E_m$ is an $N$-tiling of $\T$ which is a minimizer in
  \begin{equation}
    \label{variational problem volumes}
      \inf\big\{P(\E):|\E(h)|=m_h\quad\forall h=1,...,N\big\}
  \end{equation}
  then, up to a relabeling of the chambers of $\E_m$, there exists a $C^{1,1}$-diffeomorphism $f_m:\pa\H\to\pa\E_m$ such that
  \begin{equation}
    \label{spirit}
  \|f_m-(v+\Id)\|_{C^0(\pa\H)}^2+\|f_m-(v+\Id)\|_{C^1(\pa\H)}^4\le C_0\,\sum_{h=1}^N\,|m_h-1|\,,
  \end{equation}
  for some $v=(t\sqrt{3}\ell,s\ell)$, $s,t\in[0,1]$.
\end{theorem}

Next, let us consider the family $X$ of those $\Phi\in C^0(\T\times S^{n-1};(0,\infty))$ such that the positive one-homogeneous extension of $\Phi(x,\cdot)$ to $\R^2$ is convex, fix $\psi\in C^0(\T;(0,\infty))$, and consider the isoperimetric problem
\begin{equation}
  \label{finsler}
  \l(\Phi,\psi)=\inf\Big\{\PHI(\E)=\frac12\sum_{h=1}^N\PHI(\E(h)):\int_{\E(h)}\psi=\frac1{N}\int_\T\psi\quad \forall h=1,...,N\Big\}\,,
\end{equation}
where for a set of finite perimeter $E\subset\T$ we have set
\[
\PHI(E;A)=\int_{A\cap\pa^*E}\Phi(x,\nu_E(x))\,d\H^1(x)\,,\qquad \PHI(E)=\PHI(E;\R^n)\,,
\]
provided $\pa^*E$ and $\nu_E:\pa^*E\to S^1$ denote, respectively, the reduced boundary and the measure-theoretic outer unit normal of $E$, see \cite[Chapter 15]{maggiBOOK}. Notice that although we do not assume $\Phi$ to be even, we have nevertheless that $\l(\Phi,\psi)=\l(\hat\Phi,\psi)$ where $\hat\Phi(x,\nu)=(\Phi(x,\nu)+\Phi(x,-\nu))/2$. An interesting example is obtained when $g$ is a Riemannian metric on $\T$ and
\[
\Phi(x,\nu)=\sqrt{g(x)[\nu^\perp,\nu^\perp]}\,,\qquad \psi=\sqrt{\det(g(x))}\,,
\]
where $\nu^\perp=(\nu_2,-\nu_1)$ if $\nu=(\nu_1,\nu_2)$. In this case, \eqref{finsler} boils down to minimizing the total Riemannian perimeter of a partition of $\T$ into $N$-regions of equal Riemannian area.

\begin{theorem}\label{thm pertub metric}
  Given $L>0$ and $\g\in(0,1]$, there exist $C_0,\de_0>0$ (depending on $\T$, $L$ and $\g$) with the following property. If $\E$ is a minimizer in \eqref{finsler} for $\Phi\in X\cap\Lip(\T\times S^1)$ and $\psi\in C^{1,\g}(\T)$ such that
  \begin{gather}\label{hp finsler 1}
  \Lip\,\Phi+\|\psi\|_{C^{1,\g}(\T)}\le L\,,
  \\\nonumber
  \|\Phi-1\|_{C^0(\T\times S^1)}+\|\psi-1\|_{C^0(\T)}<\de_0\,,
  \end{gather}
  then
  \begin{equation}
    \label{venerdi}  \inf_{s,t\in[0,1]}\hd(\pa\E,v+\pa\H)^4\le C_0\,\Big(\|\Phi-\Id\|_{C^0(\T\times S^1)}+\|\psi-1\|_{C^0(\T)}\Big)\,,
  \end{equation}
  where $v=(t\sqrt{3}\ell,s\ell)$ and $\hd(S,T)$ denote the Hausdorff distance between the closed sets $S$ and $T$ in $\T$.
\end{theorem}

We deduce Theorem \ref{thm pertub metric} from Theorem \ref{thm main periodic} by some comparison arguments and density estimates. Since we are assuming that $\nabla\Phi$ is merely bounded, we do not expect $\pa\E$ to be a $C^1$-diffeomorphic image of $\pa\H$. From this point of view,  \eqref{venerdi} seems to express a qualitatively sharp control on $\pa\E$. At the same time, when more regular integrands $\Phi$ are considered (see, e.g., \cite{DuzaarSteffen} for the kind of assumption one may impose here) one would expect to be able to obtain a control in the spirit of \eqref{spirit}. However a description of singularities of isoperimetric clusters in this kind of setting, although arguably achievable at least in some special cases, is missing at present. In turn, understanding singularities would be the essential in order to adapt the improved convergence theorem from \cite{CiLeMaIC1} to this context, and thus to be able to strengthen \eqref{venerdi} into an estimate analogous to \eqref{spirit}.

The paper is organized as follows. In section \ref{section indrei} we deduce from \cite{shilleto,indreinurbekyan} a quantitative isoperimetric inequality for polygons of possible independent interest. In section \ref{section small def} we prove Theorem \ref{thm main periodic} on small $C^1$-deformations of $\pa\H$ (actually with the Hausdorff distance between $\pa\E$ and $\pa\H$ in place of $\d(\E,\H)$ on the right-hand side of \eqref{hexagonal honeycomb thm torus quantitative}). In section \ref{section fine} we exploit the improved convergence theorem from \cite{CiLeMaIC1} to deduce Theorem \ref{thm main periodic} and Theorem \ref{thm pertub volumes} from Theorem \ref{hexagonal honeycomb thm torus quantitative small}, and, finally, to deduce Theorem \ref{thm pertub metric} from Theorem \ref{thm main periodic}.

\bigskip

\noindent {\bf Acknowledgement}: The work of MC was supported by the project 2010A2TFX2 ``Calcolo delle Variazioni'' funded by the Italian Ministry of Research and University. The work of FM was supported by NSF Grant DMS-1265910.

\section{A quantitative isoperimetric inequality for polygons}\label{section indrei} Thorough this section we fix $n\ge 3$. We denote by $\Pi$ a convex unit-area $n$-gon, and by $\Pi_0$ a reference unit-area regular $n$-gon. If $\ell$ and $r$ denote, respectively, the side-length and radius of $\Pi_0$, then one easily finds that
\[
P(\Pi_0)=n\,\ell=2\sqrt{n\,\tan\Big(\frac{\pi}n\Big)}\,,\qquad r^{-1}=\sqrt{n\,\sin\Big(\frac{\pi}n\Big)\,\cos\Big(\frac{\pi}n\Big)}\,.
\]
(Notice that in the other sections of the paper we always assume $n=6$, so that $\ell=(12)^{1/4}/3$ according to the convention set in the introduction.) The isoperimetric theorem for $n$-gons asserts that
\begin{equation}
  \label{isoperimetric inequality ngon}
  P(\Pi)\ge n\,\ell\,,
\end{equation}
with equality if and only if $\Pi=\rho(\Pi_0)$ for a rigid motion $\rho$ of $\R^2$. A sharp quantitative version of \eqref{isoperimetric inequality ngon} is proved in \cite{indreinurbekyan} starting from the main result in \cite{shilleto}. Precisely, let us now denote by $\ell_i$ and $r_i$ the lengths of the $i$th edge and the $i$th radius of $\Pi$ (labeled so that $\ell_i=\ell_j$ and $r_i=r_j$ if $i=j$ modulo $n$), and set
\[
\bar{\ell}=\frac1n\sum_{i=1}^n\ell_i\,,\qquad\bar{r}=\frac1n\sum_{i=1}^n r_i\,.
\]
Then \cite[Corollary 1.3]{indreinurbekyan} asserts that
\begin{equation}
  \label{isoperimetric inequality ngon indrei}
  C(n)\,\big(P(\Pi)^2-(n\ell)^2\big)\ge \sum_{i=1}^n (r_i-\bar{r})^2+\sum_{i=1}^n (\ell_i-\bar{\ell})^2\,.
\end{equation}
The right-hand side of inequality \eqref{isoperimetric inequality ngon indrei} measures the distance of $\Pi$ from being a unit-area regular $n$-gon in the sense that if $r_i=\bar r$ and $\ell_i=\bar\ell$, then it must be $\bar r=r$ and $\bar\ell=\ell$ by the area constraint, and thus $\Pi$ is a regular unit-area $n$-gon. However, in addressing our problem we shall need (in the case $n=6$) to control the distance of $\Pi$ from a specific regular unit-area $n$-gon by means of $P(\Pi)^2-(n\ell)^2$. Passing from \eqref{isoperimetric inequality ngon indrei} to this kind of control is the subject of the following proposition.

\begin{proposition}\label{lemma indrei}
There exists a positive constant $C(n)$ with the following property: for every convex unit-area $n$-gon $\Pi$ there exists a rigid motion $\rho$ of $\R^2$ such that
   \begin{equation}
    \label{quantitative indrei inq}
      C(n)\,\big(P(\Pi)^2-(n\ell)^2\big)\ge \hd(\pa \Pi,\pa \rho\Pi_0 )^2\,.
  \end{equation}
\end{proposition}

\begin{proof}
  Up to a translation, we can assume that $\Pi$ has barycenter at $0$. Next, if $P(\Pi)\ge n\ell+\eta\,P(\Pi)$ for some $\eta>0$, then $P(\Pi)^2-(n\ell)^2\ge \eta\,P(\Pi)^2$. Since $\hd(\pa\Pi,\pa\rho\Pi_0)<\diam(\Pi)+\diam(\Pi_0)\le (P(\Pi)+P(\Pi_0))/2\le P(\Pi)$ whenever $\pa\rho\Pi_0$ intersects $\pa\Pi$, we conclude that \eqref{quantitative indrei inq} holds with $C(n)=\eta^{-1}$. In other words, in proving \eqref{quantitative indrei inq}, one can assume without loss of generality that
  \begin{equation}\label{piccolezza}
  P(\Pi)-n\,\ell< \eta\,P(\Pi)\,
  \end{equation}
  for an arbitrarily small constant $\eta=\eta(n)$. By a trivial compactness argument (on the class of convex $n$-gons with barycenter at $0$), one sees that given $\e>0$ there exists $\eta>0$ such that if \eqref{piccolezza} holds, then, up to rigid motions,
  \begin{equation}
    \label{piccolezza hd}
    \hd(\pa \Pi, \pa \Pi_0)< \e\,,
  \end{equation}
  where the reference regular unit-area $n$-gon $\Pi_0$ is assumed to have barycenter at $0$.

  Now let $v_i$ and $w_i$ denote the positions of the vertexes of $\Pi$ and $\Pi_0$ respectively: by \eqref{piccolezza hd} and up to a rotation, one can entail that
  \[
  |v_i-w_i|<\e\,,\qquad\forall i=1,...,n\,,\qquad v_1=\l\,w_1\qquad\mbox{for some $\l>0$}\,.
  \]
  Let $\rho_i$ denote the rotation around the origin such that $\rho_i(v_i)=\l_i\,w_i$ for some $\l_i>0$ (so that $\rho_1=\Id$ by $v_1=\l\,w_1$), and let $\theta_i$ denote the angle identifying $\rho_i$ as a counterclockwise rotation; since $\|\rho_i-\Id\|\le|\theta_i|$ and $|\rho_i(v_i)-w_i|=|r_i-r|$, one has
  \begin{equation}
    \label{bene -1}
    \hd(\pa\Pi,\pa\Pi_0)\le C\,\sum_{i=1}^n\,|v_i-w_i|\le C\,\sum_{i=1}^n\,r_i|\theta_i|+|r_i-r|\,.
  \end{equation}
  Let us now set $\de=P(\Pi)-n\ell$: by \eqref{isoperimetric inequality ngon indrei} and \eqref{piccolezza} one finds
  \begin{equation}
    \label{bene 0}
  \max_{1\le i\le n}|r_i-\bar{r}|+|\ell_i-\bar\ell|\le C\,\sqrt{\de}\,.
  \end{equation}
  Since $\bar\ell=n^{-1}P(\Pi)$ gives $|\bar\ell-\ell|=n^{-1}\de$, we deduce from $|\ell_i-\bar\ell|\le C\sqrt\de$ that
  \begin{equation}
    \label{bene 1}
      \max_{1\le i\le n}|\ell_i-\ell|\le C\sqrt{\de}\,.
  \end{equation}
  Let now $A(a,b,c)$ denote the area of a triangle with sides of length $a$, $b$ and $c$. Since $A$ is a Lipschitz function in an $\e$-neighborhood of $(r,r,\ell)$ (where both $(\bar r,\bar r,\ell)$ and $(r_i,r_{i+1},\ell_i)$ lie by \eqref{piccolezza hd}), by \eqref{bene 0}, \eqref{bene 1} and by $|\Pi_0|=|\Pi|$ we find
  \[
  \Big|n\,A(r,r,\ell)-n\, A(\bar{r},\bar{r},\ell)\Big|
  =\Big|\sum_{i=1}^n A(r_i,r_{i+1},\ell_i)-n\, A(\bar{r},\bar{r},\ell)\Big|\le C\,\sqrt{\de}\,.
  \]
  Since $A(a,a,\ell)=(\ell/4)\,\sqrt{4a^2-\ell^2}$ we immediately see that $|A(r,r,\ell)-A(a,a,\ell)|\ge c\,|a-r|$ whenever $|a-r|<\e$ and where $c=c(\ell)=c(n)>0$. Thus, $|r-\bar{r}|\le C\,\sqrt{\de}$, and  \eqref{bene 0} and \eqref{bene 1} give
  \begin{equation}
    \label{bene 2}
    \max_{1\le i\le n}|r_i-r|+|\ell_i-\ell|\le C\,\sqrt{\de}\,.
  \end{equation}
  If $\a_i$ denotes the interior angle between $v_i$ and $v_{i+1}$ (so that $|\a_i-2\pi/n|=O(\e)$ by \eqref{piccolezza hd}), then
  \[
  \a_i=f(r_i,r_{i+1},\ell_i)\,,\qquad\mbox{where}\quad f(a,b,c)=\arccos\Big(\frac{a^2+b^2-c^2}{2ab}\Big)\,.
  \]
  Since $f$ is a Lipschitz function in an $\e$-neighborhood of $(r,r,\ell)$, we conclude from \eqref{bene 2} that
  \[
  \max_{1\le i\le n}\big|\a_i-\frac{2\pi}n\big|=\max_{1\le i\le n}\big|f(r_i,r_{i+1},\ell_i)-f(r,r,\ell)\big|\le C\,\sqrt\de\,.
  \]
  In particular, since $\theta_1=0$ (as $\rho_1=\Id$), we deduce from this last estimate that $|\theta_i|\le C\,\sqrt{\de}$ for $i=1,...,n$.
  We plug this inequality and \eqref{bene 2} in \eqref{bene -1} to conclude the proof.
\end{proof}

Coming to the torus $\T$, we shall use the following corollary of Proposition \ref{lemma indrei}.

\begin{corollary}\label{corollary indrei}
  There exist positive constants $\eta$ and $c$, independent from $\T$, with the following property. If $\Pi$ is a convex hexagon in $\T$ such that $\hd(\pa\Pi,\pa H)\le \eta$, then there exists a regular hexagon $H_*$ in $\T$ with $|\Pi|=|H_*|$
  \begin{equation}
    \label{precisi0}
      P(\Pi)-P(H)\sqrt{|\Pi|}\ge c\,\hd(\pa\Pi,\pa H_*)^2\,.
  \end{equation}
\end{corollary}

\begin{proof}
  We first notice that by Proposition \ref{lemma indrei} and by scaling, if $\hat\Pi$ is a convex hexagon in $\R^2$, then there exists a regular hexagon $\hat H_*$ with $|\hat H_*|=|\hat\Pi|$ and
  \begin{equation}
    \label{precisi1}
      P(\hat\Pi)^2-P(\hat H)^2 |\hat\Pi|\ge c\,\hd(\pa\hat\Pi,\pa \hat H_*)^2\,.
  \end{equation}
  Since $\Pi$ is a convex hexagon in $\T$ with $\hd(\pa\Pi,\pa H)\le\eta$, then there exists a convex hexagon $\hat\Pi$ in $\R^2$ isometric to $\Pi$ with $\hd(\pa\hat\Pi,\pa\hat H)\le\eta$. In particular, for some constant $C$ independent from $\T$, one has
  \[
  P(\hat\Pi)-P(\hat H)\sqrt{|\hat\Pi|}\le C\,\eta\,,\qquad P(\hat\Pi)+P(\hat H)\sqrt{|\hat\Pi|}\le C\,,
  \]
  and thus \eqref{precisi1} gives, up to further decrease the value of $c$,
  \begin{equation}
    \label{precisi2}
      C\eta\ge P(\hat\Pi)-P(\hat H)\sqrt{|\hat\Pi|}\ge c\,\hd(\pa\hat\Pi,\pa \hat H_*)^2\,.
  \end{equation}
  By \eqref{precisi2} and $\hd(\pa\hat\Pi,\pa\hat H)\le\eta$ we have $\hd(\pa\hat H,\pa\hat H_*)\le C\sqrt\eta$. Now, since $\beta\ge 2$ and $\a$ is even one can find $\eta_*>0$ (independent of $\a$ and $\beta$) such that $I_{\eta_*}(\hat H)=\{x\in\R^2:\dist(x,\hat{H})\le\eta_*\}$ is compactly contained into a rectangular box of height $3\ell\a/2$ and width $\sqrt{3}\ell\beta$. As a consequence, if $\hat{J}$ is a polygon contained in $I_{\eta_*}(\hat H)$, then $J=\hat{J}/_\approx\subset\T$ is isometric to $\hat J$. Thus, if $C\sqrt\eta<\eta_*$, then $H_*=\hat H_*/_\approx$ is a regular hexagon in $\T$ with $|H_*|=|\Pi|$ and $\hd(\pa\hat\Pi,\pa\hat H_*)=\hd(\pa\Pi,\pa H_*)$, and \eqref{precisi0} follows from \eqref{precisi2}.
\end{proof}

\section{Small deformations of the reference honeycomb}\label{section small def} The main result of this section is Theorem \ref{hexagonal honeycomb thm torus quantitative small}, which provides us, on a restricted class of unit-area tilings, with a stronger stability estimate than the one in Theorem \ref{thm main periodic}. Before stating this result we need to introduce the following terminology:

\medskip

\noindent {\bf Regular and singular sets:} Given a $N$-tiling $\E$ of $\T$ one sets
\[
\pa\E=\bigcup_{h=1}^N\pa\E(h)\,,\qquad\pa^*\E=\bigcup_{h=1}^N\pa^*\E(h)\,,
\]
\[
\S(\E)=\pa\E\setminus\pa^*\E\,,
\qquad[\pa\E]_\mu=\{x\in\pa\E:\dist(x,\S(\E))>\mu\}\,,\quad\mu>0\,,
\]
where $\pa^*E$ denotes the reduced boundary of a set of finite perimeter $E$ in $\T$, and where the normalization convention $\pa E=\ov{\pa^*E}$
for sets of finite perimeter is always assumed to be in force, see \cite[Section 12.3]{maggiBOOK}. We call $\pa^*\E$ and $\S(\E)$ the {\it regular set} and the {\it singular set} of $\pa\E$ respectively. In this way, $\pa^*\H$ and $\S(\H)$ are, respectively, the union of the open edges and the union of the vertexes of the hexagons $\H(h)$ for $h=1,...,N$.

\medskip

\noindent {\bf Tilings and maps of class $C^{k,\a}$:} Given $k\in\N$ and $\a\in[0,1]$, one says that a tiling $\E$ of $\T$ is of {\it class $C^{k,\a}$} if there exist a finite family $\{\g_i\}_{i\in I}$ of compact $C^{k,\a}$-curves with boundary and a finite family $\{p_j\}_{j\in J}$ of points such that
\begin{equation}
  \label{class C1}
  \pa\E=\bigcup_{i\in I}\g_i\,,\qquad\pa^*\E=\bigcup_{i\in I}\INT(\g_i)\,,\qquad\S(\E)=\bigcup_{i\in I}\bd(\g_i)=\bigcup_{j\in J}\{p_j\}\,,
\end{equation}
where $\INT(\g_i)$ and $\bd(\g_i)$ denote the interior and the boundary of $\g_i$ respectively. Moreover, given a function $f:\pa\E\to\T$, one says that $f\in C^{k,\a}(\pa\E;\T)$ if $f$ is continuous on $\pa\E$ and
\[
\|f\|_{C^{k,\a}(\pa\E)}:=\sup_{i\in I}\|f\|_{C^{k,\a}(\g_i)}<\infty\,.
\]
Finally, given two $C^{k,\a}$-tilings $\E$ and $\F$ of $\T$, one says that $f$ is a $C^{k,\a}$-diffeomorphism between $\pa\E$ and $\pa\F$ if $f$ is an homeomorphism between $\pa\E$ and $\pa\F$ with $f(\S(\E))=\S(\F)$, $f(\pa\E(h))=\pa\F(h)$ for every $h=1,...,N$, $f\in C^{k,\a}(\pa\E;\T)$ and $f^{-1}\in C^{k,\a}(\pa\E;\T)$.

\medskip

\noindent {\bf Tangential component of a map and $(\e,\mu,L)$-perturbations of $\H$:} Given a tiling $\E$ of $\T$ of class $C^1$, by taking \eqref{class C1} into account one can define $\nu_\E\in C^{0}(\pa^*\E;S^1)$ in such a way that $\nu_\E$ is a unit normal vector to $\g_i$ for every $i$. Correspondingly, given a map $f:\pa\E\to\T$, we define $\ttau_\E f:\pa^*\E\to\T$, the tangential component of $f$ with respect to $\pa\E$, as
\[
\ttau_\E f(x)=f(x)-(f(x)\cdot\nu_\E(x))\,\nu_\E(x)\,,\qquad x\in\pa^*\E\,.
\]
Finally, one says that $\E$ is an {\it $(\e,\mu,L)$-perturbation of $\H$} if $\E$ is of class $C^{1,1}$ and there exists an homeomorphism $f$ between $\pa\H$ and $\pa\E$ with
\begin{eqnarray}\label{eL perturbation 2}
  \|f\|_{C^{1,1}(\pa\H)}&\le&L\,,
  \\\nonumber
  \|f-\Id\|_{C^1(\pa\H)}&\le&\e\,,
  \\\nonumber
  \|\ttau_{\H}(f-\Id)\|_{C^1(\pa^*\H)}&\le&\frac{L}\mu\,\sup_{\S(\H)}|f-\Id|\,,
    \\\nonumber
  \ttau_{\H}(f-\Id)=0\,,&&\qquad\mbox{on $[\pa\H]_\mu$}\,.
\end{eqnarray}

\begin{theorem}
  \label{hexagonal honeycomb thm torus quantitative small}
  For every $L>0$ there exist positive constants $\mu_0$, $\e_0$ and $c_0$ (depending on $L$ and $|\T|$), $C$ depending on $|\T|$ only, and $C'$ depending on $L$ only, with the following property. If $\E$ is a unit-area $(\e_0,\mu_0,L)$-perturbation of $\H$, then there exists $v\in\R^2$ such that
  \begin{equation}
    \label{stabilita hd}
      P(\E)-P(\H)\ge c_0\,\hd(\pa\E,v+\pa\H)^2\,,\qquad |v|\le C\,\e_0\,.
  \end{equation}
  Moreover, there exists a $C^{1,1}$-diffeomorphism $f_0$ between $v+\pa\H$ and $\pa\E$ such that
  \begin{equation}
    \label{stabilita f C1}
      P(\E)-P(\H)\ge c_0\,\Big(\|f_0-\Id\|_{C^0(v+\pa\H)}^2+\|f_0-\Id\|_{C^1(v+\pa\H)}^4\Big)\,,
  \end{equation}
  and $\|f_0\|_{C^{1,1}(v+\pa\H)}\le C'$.
\end{theorem}

We premise a lemma to the proof of Theorem \ref{hexagonal honeycomb thm torus quantitative small}. As said, this lemma provides a quantitative version of (a particular case of) Hales' hexagonal isoperimetric inequality, the key step in the proof of \eqref{hexagonal honeycomb thm torus} in \cite{hales}.

\begin{lemma}\label{lemma hales}
  There exist positive constants $\e_1$ and $c_1$ with the following property. If $\E$ is a unit-area tiling of $\T$ such that there exists an homeomorphism $f$ between $\pa\H$ and $\pa\E$ with $\|f-\Id\|_{C^0(\pa\H)}\le\e_1$, if $E=\E(h)$ for some $h\in\{1,...,N\}$ and $\Pi$ is the convex envelope of $\S(\E)\cap\pa E$ (so that $\Pi$ is convex hexagon with set of vertexes $\S(\E)\cap\pa E$ provided $\e_1$ is small enough), then there exists a regular hexagon $H_*$ with $|H_*|=|\Pi|$ such that
  \begin{equation}
    \label{utile}
    P(E)\ge P(H)+\frac{P(H)}2\,(|\Pi|-|E|)+c_1\,\Big(|E\Delta\Pi|^2+\hd(\pa\Pi,\pa H_*)^2\Big)\,.
  \end{equation}
\end{lemma}

\begin{remark}
  {\rm The constants $\e_1$ and $c_1$ will just depend on the metric properties of the unit-area hexagon. In particular they do not depend on $\T$.}
\end{remark}

\begin{proof}
  [Proof of Lemma \ref{lemma hales}]
  Let $\arc_t(a)$ denote the length of a circular arc that bounds an area $a\ge0$ and whose chord length is $t>0$, and let us set $\arc(a)=\arc_1(a)$. In this way, $\arc:[0,\infty)\to[1,\infty)$ is an increasing function. Since the derivative of $\arc$ at $a$ is the curvature of any circular arc bounding an area $a$ above a unit length chord, and since this curvature is increasing as $a$ ranges from $0$ to $\pi/8$ (the value $a=\pi/8$ corresponds to the case of an half-disk with unit diameter), we conclude that $\arc$ is convex on $[0,\pi/8]$ (and, in fact, also concave on $[\pi/8,\infty)$). Moreover, a Taylor expansion gives that $\arc''(0^+)>0$: hence there exists $\eta>0$ such that
  \begin{equation}
    \label{arc coercivo}
    \arc(a)\ge 1+\eta\,a^2\,,\qquad\forall a\in[0,\eta)\,.
  \end{equation}
  Let $\ell_i$ denote the length of the $i$th side of $\Pi$, and let $a_i$ denote the total area enclosed between the $i$th side of $\Pi$ and the $i$th side of $E$;
  \begin{figure}
    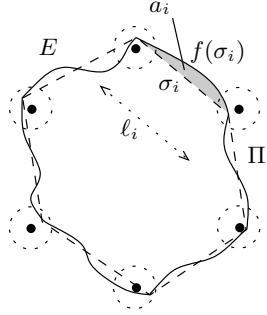\caption{{\small The convex hexagon $\Pi$ spanned by $\S(\E)\cap\pa E$. The vertexes of $\Pi$ are $\e_1$-close to the vertexes of the unit-area regular hexagon $\H(h)$ (as $E=\E(h)$ and $f(\pa\H(h))=\pa\E(h)$) which are depicted as black dots. The boundaries of $\Pi$ and $E$ are depicted, respectively, by a dashed line and by a continuous line.}}\label{fig pi}
  \end{figure}
  see Figure \ref{fig pi}. (If $\s_i$ is the $i$th side of $\Pi$, then the $i$th side of $E$ is a small $C^0$-deformation of $\s_i$ with fixed end-points). Noticing that $\arc_t(a)=t\,\arc(a/t^2)$, by Dido's inequality we find that
  \[
  P(E)\ge\sum_{i=1}^6\arc_{\ell_i}(a_i)=\sum_{i=1}^6\,\ell_i\,\arc\Big(\frac{a_i}{\ell_i^2}\Big)\,.
  \]
  By $\|f-\Id\|_{C^0(\pa\H)}\le \e_1$ and provided $\e_1\le 1$, one has
    \begin{equation}
    \label{omega}
    \hd(\pa\Pi,\pa\H(h))\le \e_1\,,\qquad \max_{1\le i\le 6}\Big\{a_i,\Big|\ell_i-\frac{P(H)}6\Big|\Big\}\le C\,\e_1\,,
  \end{equation}
  where a possible value for $C$ in \eqref{omega} is $2(\pi+\ell)$. By \eqref{omega}, by further decreasing $\e_1$, we can assume that $a_i/\ell_i^2\in[0,\pi/8]$ for every $i=1,...,6$. We thus apply Jensen inequality to find that
  \[
  P(E)\ge \sum_{i=1}^6\,\ell_i\,\arc\bigg(\frac1{\sum_{i=1}^6\ell_i}\sum_{i=1}^6\frac{a_i}{\ell_i^2}\bigg)\,.
  \]
  Since $P(H)/6=(12)^{1/4}/3<1$, by \eqref{omega} we may further assume that $\ell_i\le 1$ for every $i=1,...,6$, and thus conclude by $P(\Pi)=\sum_{i=1}^6\ell_i$, $|E\Delta\Pi|=\sum_{i=1}^6a_i$, and the monotonicity of $\arc$ that
  \begin{equation}
    \label{chordal isoperimetric inequality}
    P(E)\ge P(\Pi)\,\arc\Big(\frac{|E\Delta \Pi|}{P(\Pi)}\Big)\,.
  \end{equation}
  (Inequality \eqref{chordal isoperimetric inequality} is clearly related to the {\it chordal isoperimetric inequality} \cite[Proposition 6.1-A]{hales}, see also \cite[15.5]{Morgan}.) By \eqref{arc coercivo}, \eqref{omega} and \eqref{chordal isoperimetric inequality},
  \begin{equation}
    \label{chordal isoperimetric inequality x}
    P(E)\ge P(\Pi)+\eta\,\frac{|E\Delta\Pi|^2}{P(\Pi)^2}\ge P(\Pi)+c_1\,|E\Delta\Pi|^2\,,
  \end{equation}
  where $c_1>0$. Provided $\e_1$ is small enough, by \eqref{omega} we can apply Corollary \ref{corollary indrei} to find a regular hexagon $H_*$ with $|H_*|=|\Pi|$ and
  \[
    P(\Pi)- P(H)\sqrt{|\Pi|}\ge c\,\hd(\pa\Pi,\pa H_*)^2\,.
  \]
  Thus, up to further decrease the value of $c_1$,  \eqref{chordal isoperimetric inequality x} gives
  \begin{equation}
    \label{proof1}
    P(E)\ge P(H)\sqrt{|\Pi|}+c_1\,\Big(\hd(\pa\Pi,\pa H_*)^2+|E\Delta\Pi|^2\Big)\,.
  \end{equation}
  Finally, given $\tau>0$ let $\l>0$ be such that $\sqrt{1-s}\ge1-(s/2)-\tau\,s^2$ for $|s|<\l$: up to further decrease $\e_1$, by $\|f-\Id\|_{C^0(\pa\H)}\le\e_1$ we entail $|\s|<\l$ for $\s=|E|-|\Pi|$, and thus deduce with the aid of \eqref{proof1}  and $|E|=1$ that
  \begin{equation}
    \label{proof3}
    P(E)\ge P(H)-\frac{P(H)}2\,\s-P(H)\tau\,\s^2
    +c_1\,\Big(\hd(\pa\Pi,\pa H_*)^2+|E\Delta\Pi|^2\Big)\,.
  \end{equation}
  Since $|\s|=||E|-|\Pi||\le|E\Delta\Pi|$, for $\tau$ small enough depending from $c_1$, we prove \eqref{utile}.
\end{proof}

\begin{proof}[Proof of Theorem \ref{hexagonal honeycomb thm torus quantitative small}] {\it Step one}: The reflection of $\R^2$ with respect to a generic line does not induce a map on $\T$. However, by \eqref{occhio}, one has that if $R_\theta\hat H$ denotes the counterclockwise rotation of $\hat H$ by an angle $\theta$ around the origin, then $R_\theta\hat H$ is compactly contained in a box of height $3\ell\a/2\ge3\ell$ and width $\sqrt{3}\ell\beta\ge2\sqrt{3}\ell$ for every $\theta$. As a consequence, given a unit-area regular hexagon $K$ in $\T$, all the rotations of $K$ are well-defined as unit-area regular hexagons in $\T$; in particular, it always makes sense to define the reflection $g_\s(K)$ of $K$ with respect to an edge $\s$ of $K$. Taking this into account, we notice that there exist positive constants $\eta$ and $C$ (independent of $\T$) such that, if $K$ and $K'$ are unit-area regular hexagons in $\T$, and if $\s$ and $\s'$ are edges of $K$ and $K'$ respectively, then
\[
\left\{
\begin{array}
  {l l}
  \hd(\s,\s')\le\eta\,,
  \\
  |K\Delta K'|\ge 2-\eta\,,
\end{array}
\right .
\qquad\Rightarrow\qquad
\hd(\pa g_\s(K),\pa K')\le C\,\hd(\s,\s')\,.
\]
This geometric remark is going to be repeatedly used in the following arguments, where we shall denote by $\e_1$ and $c_1$ the constants of  Lemma \ref{lemma hales} and set $\de=P(\E)-P(\H)$. We notice that, by the area formula and since $\|f-\Id\|_{C^1(\pa\H)}\le\e_0$, one has
\begin{equation}
  \label{delta eps0}
  \de\le C\,P(\H)\,\e_0^2\,,
\end{equation}
where $C$ is independent from $\T$ and where $P(\H)=|\T|\,P(H)/2$.

\medskip

\noindent  {\it Step two}: We claim that, if $\e_0$ is small enough depending only from $|\T|$, and if $\Pi_h$ denotes the convex envelope of $\pa\E(h)\cap\S(\E)$ (so that $\Pi_h$ is a convex hexagon, not necessarily with unit-area), then for every $h=1,...,N$ there exists a regular unit-area hexagon $K_h$ such that
\begin{eqnarray}
  \label{trombone}
  \hd(\pa\Pi_h,\pa K_h)\le C\,\sqrt\de\,,&&
  \\
  \label{trombone 2}
  |K_h\Delta K_{h+1}|\ge 2-C\,\sqrt\de\,,&&
\end{eqnarray}
where here and in the rest of this step, $C$ denotes a constant depending from $|\T|$ only. Indeed, since $\{\Pi_h\}_{h=1}^N$ is a partition of $\T$, one has $\sum_{h=1}^N|\Pi_h|=|\T|=\sum_{h=1}^N|\E(h)|$. By requiring $\e_0\le\e_1$ we can apply Lemma \ref{lemma hales} to each $\E(h)$ in order to find regular hexagons $H^*_h$ with $|H^*_h|=|\Pi_h|$ such that, by adding up \eqref{utile} on $h$, one finds
\begin{eqnarray}\label{pino}
  2\,\de=\sum_{h=1}^N (P(\E(h))-P(H))\ge c_1\,\sum_{h=1}^N\Big(|\E(h)\Delta\Pi_h|^2+\hd(\pa\Pi_h,\pa H^*_h)^2\Big)\,.
\end{eqnarray}
By \eqref{pino},
\begin{equation}
  \label{esagoniiii}
  ||\Pi_h|-1|\le|\E(h)\Delta \Pi_h|\le \sqrt{\frac{2\de}{c_1}}\,.
\end{equation}
By \eqref{occhio}, we may further decrease the value of $\eta$ introduced in step one so to have that if $J$ is a regular hexagon in $\T$ with $||J|-1|\le\eta$, then it makes sense to scale $J$ with respect to its barycenter in order to obtain a unit-area regular hexagon $J'$ with $\hd(\pa J,\pa J')\le C\,||J|-1|$. In particular, by \eqref{delta eps0} and \eqref{esagoniiii}, up to decrease the value of $\e_0$ we can define unit-area hexagons $K_h$ in $\T$ with the property that
\[
\hd(\pa K_h,\pa H_h^*)\le C\,||H_h^*|-1|=C\,||\Pi_h|-1|\le C\,\sqrt\de\,.
\]
By combining this estimate with \eqref{pino} we prove \eqref{trombone}. By \eqref{trombone}, $|K_j\Delta\Pi_j|\le C\sqrt\de$ for every $j$, and thus
\[
|K_h\Delta K_{h+1}|\ge|\E(h)\Delta\E(h+1)|-\sum_{j=h}^{h+1}|\E(j)\Delta K_j|\ge 2-C\,\sqrt\de-\sum_{j=h}^{h+1}|\E(j)\Delta \Pi_j|\,.
\]
In particular, \eqref{trombone 2} follows from \eqref{pino}.

\medskip

\noindent {\it Step three}: We claim the existence of a tiling $\H_0=v+\H$ of $\T$ such that
\begin{equation}
  \label{bella li}
  \hd(\S(\E),\S(\H_0))\le C\,\sqrt\de\,,\qquad |v|\le C\,\e_0\,,
\end{equation}
where here and in the rest of this step, $C$ denotes a constant depending from $|\T|$ only. Let us recall from Figure \ref{fig hexagon} that the chambers of $\H$ are ordered so that $\{\H(h)\}_{h=1}^\b$ is the ``bottom row'' in the grid defined by $\H$ and that $\H(1)=H$. Since $\E$ is an $(\e_0,\mu_0,L)$-perturbation of $\H$ one has
\begin{equation}
  \label{vicine}
  \max\Big\{\hd(\pa\E(h),\pa\H(h)),\hd(\pa\Pi_h,\pa\H(h))\Big\}\le \e_0\,,\qquad\forall h=1,...,N\,,
\end{equation}
so that \eqref{trombone} implies $\hd(\pa H,\pa K_1)\le C\,\e_0$. In particular, there exists $|\theta|,|s|,|t|\le C\e_0$ such that
\[
K_1=(t\sqrt{3}\ell,s\ell)+R_\theta H\,,
\]
where, with a slight abuse of notation, $R_\theta H$ denotes the counterclockwise rotation of $H$ by an angle $\theta$ around its left-bottom vertex (see step one). Of course, there is no reason to get a better estimate than $|s|,|t|\le C\,\e_0$ here (indeed, $\E$ itself could just be an $\e_0$-size translation of $\H$). Nevertheless, if $\theta\ne 0$, then we cannot fit $K_1$ into an hexagonal honeycomb of $\T$: therefore one expects
\begin{equation}
  \label{expectations}
  |\theta|\le C\sqrt\de\,.
\end{equation}
We prove \eqref{expectations}: set $J_1=K_1$, let $\tau_1$ be the common edge between $\Pi_1$ and $\Pi_2$, and let $\s_1$ and $\s_1'$ be the edges of $K_1$ and $K_2$ respectively such that, thanks to \eqref{trombone}, $\hd(\tau_1,\s_1)+\hd(\tau_1,\s'_1)\le C\,\sqrt\de$. In this way $\hd(\s_1,\s_1')\le C\,\sqrt\de$, and by \eqref{trombone 2} we can apply step one to deduce
\begin{equation}
  \label{hey}
  \hd(\pa J_2,\pa K_2)\le C\,\hd(\s_1,\s_1')\le C\sqrt\de\,,\qquad |J_2\Delta K_2|\le C\,\sqrt\de\,,
\end{equation}
where $J_2$ is the reflection of $J_1$ with respect to $\s_1$. Let now $\tau_2$ be common side between $\Pi_2$ and $\Pi_3$. By \eqref{trombone} and \eqref{hey} we have $\hd(\pa J_2,\pa\Pi_2)+\hd(\pa K_3,\pa\Pi_3)\le C\sqrt\de$, thus there exist edges $\s_2$ and $\s_2'$ of $J_2$ and $K_3$ respectively such that $\hd(\tau_2,\s)+\hd(\tau_2,\s')\le C\,\sqrt\de$. By \eqref{trombone 2} and \eqref{hey} one has $|J_2\Delta K_3|\ge 2-C\sqrt\de$, so that by step one $\hd(\pa J_3,\pa K_3)\le C\sqrt\de$ where $J_3$ is the reflection of $J_2$ with respect to $\s_2$. If we repeat this argument $\beta$-times, then we find regular unit-area hexagons $J_1,...,J_\beta$ such that $J_1=K_1$, $J_h$ is obtained by reflecting $J_{h-1}$ with respect to its ``vertical'' right edge, and $\hd(\pa J_h,\pa K_h)\le C\,\sqrt\de$ for $h=1,...,\beta$. By construction, $\Pi_\beta$ and $\Pi_1$ also share a common edge $\tau$, and correspondingly $J_\beta$ and $K_1$ have edges $\s$ and $\s'$ respectively with $\hd(\tau,\s)+\hd(\tau,\s')\le C\,\sqrt\de$. By reflecting $J_\beta$ with respect to $\s$ we thus find a regular unit area hexagon $J_*$ with
\[
\hd(\pa J_*,\pa K_1)\le C\,\sqrt\de\,.
\]
At the same time, since $J_*$ has been obtained by iteratively reflecting $J_1=K_1$ with respect to its ``vertical'' right edge, we find that
\[
\hd(\pa J_*,\pa J_1)\ge \frac{|\theta|}C\,.
\]
Thus \eqref{expectations} holds. As a consequence, up to apply to $K_1$ a rotation of size $C\,\sqrt\de$, one can assume that
\begin{equation}
  \label{richiuditi}
  K_1=(t\sqrt{3}\ell,s\ell)+H\,,\qquad\mbox{for some $|t|,|s|\le C\,\e_0$}\,.
\end{equation}
In particular, if we set $\H_0(h)=(t\sqrt{3}\ell,s\ell)+\H(h)$, then $\H_0$ defines a unit-area tiling of $\T$ by regular hexagons. By arguing as in the proof of \eqref{expectations}, one easily sees that
\begin{equation}
  \label{rigidita}
  \hd(\pa\Pi_h,\pa \H_0(h))\le C\,\sqrt{\de}\,,\qquad\forall h=1,...,N\,.
\end{equation}
In particular, the set of vertexes of $\Pi_h$ and $\H_0(h)$ lie at distance $C\,\sqrt\de$. Since $\S(\E)$ is the set of all the vertexes of the $\Pi_h$s, we complete the proof of \eqref{bella li}.

%
%
%

\medskip

\noindent {\it Step four}: We show that if $\mu_0$ is small enough with respect to $L$, and $\e_0$ is small enough with respect to $\mu_0$ and $|\T|$, then there exists a $C^{1,1}$-diffeomorphism $f_0$ between $\pa\H_0$ and $\pa\E$ such that
\begin{equation}
  \label{friday00}
  \|f_0\|_{C^{1,1}(\pa\H_0)}\le C\,,\qquad\|f_0-\Id\|_{C^1(\pa\H_0)}\le C\,\mu_0\,,
\end{equation}
\begin{equation}
  \label{friday01}
  \|(f_0-\Id)\cdot\tau_0\|_{C^1(\pa\H_0)}\le C\,\sup_{\S(\H_0)}|f_0-\Id|\,.
\end{equation}
where $C$ depends on $L$ only.  The map $f_0$ is more useful than the map $f$ appearing in \eqref{eL perturbation 2} because the best estimate for $f-\Id$ on $\S(\H)$ is of order $\e_0$, while, thanks to \eqref{bella li}, we have a much more precise information about $f_0-\Id$ on $\S(\H_0)$, namely
\begin{equation}
  \label{deficit 0}
    \sup_{\S(\H_0)}|f_0-\Id|\le C\,\sqrt\de\,.
\end{equation}
(In \eqref{deficit 0}, $C$ depends on $|\T|$.) Let us also notice that we cannot just define $f_0$ by composing $f$ with the translation bringing $\pa\H_0$ onto $\pa\H$, because this translation is $O(\e_0)$, and thus the resulting map $f_0$ would still have tangential displacement $O(\e_0)$. We thus need a more precise construction, directly relating $\pa\H_0$ and $\pa\E$.

To this end, we fix an edge $\s$ of $\H$, and set $\s_0=v+\s$, so that $\s_0$ is an edge of $\H_0$. We denote by $\tau_0$ and $\nu_0=\tau_0^\perp$ the constant tangent and normal unit-vector fields to $\s_0$ (and, obviously, to $\s$). We let $\gamma=f(\s)$ and set $\tau(x)=\nabla^\s f(f^{-1}(x))[\tau_0]$ and $\nu(x)=\tau(x)^\perp$, where $\nabla^\s f$ denotes the tangential gradient of $f$ with respect to $\s$. Finally, we set $[\s_0]_t=\{x\in\s_0:\dist(x,\bd(\s_0))>t\}$ for $t>0$. By \cite[Theorem 2.6, Proposition B.2]{CiLeMaIC1}, given $M>0$ there exist positive constants $C_1$ and $\mu_1$ (depending on $M$ and $\s_0$) such that if, for some $\rho\le \mu_1^2$, $\g$ satisfies the following properties
  \begin{itemize}
  \item[(a)] $\hd(\s_0,\g)+\hd(\bd(\s_0),\bd(\g))\le\rho$;
  \item[(b)] $|\tau(p)-\tau_0|+|\tau(q)-\tau_0|\le\rho$ where $\{p,q\}=f(\bd(\s))$;
  \item[(c)] there exists a map $\psi_0\in C^{1,1}([\s_0]_{\rho})$ such that
  \begin{equation*}
  [\g]_{3\rho}\subset(\Id+\psi_0\nu_0)\big([\s_0]_{\rho}\big)\subset\g\,,
  \end{equation*}
  \begin{equation*}
  \|\psi_0\|_{C^{1,1}([\s_0]_{\rho})}\le M\,,\qquad \|\psi_0\|_{C^1([\s_0]_{\rho})}\le \rho\,;
  \end{equation*}
  \item[(d)] $|\nu(x)-\nu(y)|\le M\,|x-y|$ and $|\nu(x)\cdot(y-x)|\le M\,|x-y|^2$ for every $x,y\in \g$\,,
\end{itemize}
then, there exists a $C^{1,1}$-diffeomorphism $f_0$ between $\s_0$ and $\g$ such that $f_0(\bd(\s_0))=\bd(\g)$ and
\[
  \|f_0\|_{C^{1,1}(\s_0)}\le C_1\,,\qquad\|f_0-\Id\|_{C^1(\s_0)}\le \frac{C_1}{\mu_1}\,\rho\,,
\]
\[
\|(f_0-\Id)\cdot\tau_0\|_{C^1(\s_0)}\le \frac{C_1}{\mu_1}\,\sup_{\bd(\s_0)}|f_0-\Id|\,.
\]
(Since $\s_0$ is just a segment of fixed length $\ell=(12)^{1/4}/3$, we shall not stress the dependence of $C_1$ and $\mu_1$ on $\s_0$.) We notice that property (a) holds provided $\rho\ge C\e_0$ for some $C$ depending on $|\T|$ only: indeed, by $\|f-\Id\|_{C^0(\s)}\le\e_0$ one finds $\hd(\s,\g)+\hd(\bd(\s),\bd(\g))\le\e_0$, while $|v|\le C\,\e_0$ (recall \eqref{bella li}) gives $\hd(\s,\s_0)\le C\,\e_0$. Similarly, property (b) holds if $\rho\ge\e_0$, as $\tau(x)=\nabla^\s f(f^{-1}(x))[\tau_0]$ and $\|f-\Id\|_{C^1(\s)}\le \e_0$. Property (d) follows easily from $\|f\|_{C^{1,1}(\s)}\le L$ and $\|f-\Id\|_{C^1(\s)}\le\e_0$ with $M=M(L)$. Finally, concerning property (c), we notice that by exploiting the fact that $\E$ is an $(\e_0,\mu_0,L)$-perturbation of $\H$ and setting $\psi=(f-\Id)\cdot\nu_0$, one has $\psi\in C^{1,1}([\s]_{\mu_0})$ with
\begin{equation}\label{friday1}
[\g]_{\mu_0+2\e_0}\subset(\Id+\psi\nu_0)\big([\s]_{\mu_0}\big)\subset\g\,,
\end{equation}
\begin{equation}\label{friday2}
\|\psi\|_{C^{1,1}([\s]_{\mu_0})}\le L\,,\qquad \|\psi\|_{C^1([\s]_{\mu_0})}\le \e_0\,,
\end{equation}
where the first inclusion in \eqref{friday1} follows from $\|f-\Id\|_{C^0(\s_0)}\le\e_0$ and $\g=f(\s)$.
By exploiting \eqref{friday1}, \eqref{friday2}, and the fact that $\s_0=v+\s$ with $|v|\le C\,\e_0$ by \eqref{bella li}, one can find two constants $C_2\le C_3$ (both depending just on $|\T|$) and $\psi_0\in C^{1,1}([\s]_{\mu_0+C_2\,\e_0})$ such that properties (a), (b) and (d) hold with $\rho=\mu_0+C_2\,\e_0$, and
  \begin{equation}\label{friday3}
  [\g]_{\mu_0+C_3\,\e_0}\subset(\Id+\psi_0\nu_0)\big([\s_0]_{\mu_0+C_2\,\e_0}\big)\subset\g\,,
  \end{equation}
  \begin{equation}\label{friday4}
  \|\psi_0\|_{C^{1,1}([\s_0]_{\mu_0+C_2\,\e_0})}\le L\,,\qquad \|\psi_0\|_{C^1([\s_0]_{\mu_0+C_2\,\e_0})}\le \e_0\,,
  \end{equation}
see
\begin{figure}
  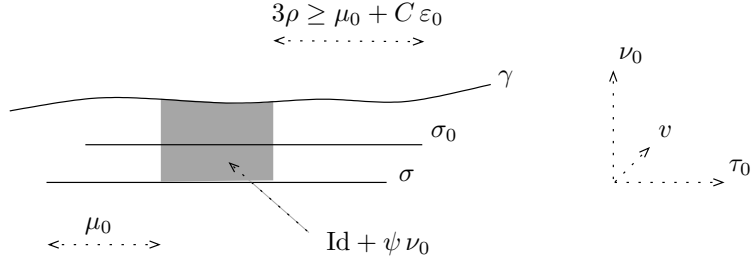\caption{{\small The function $\psi_0$ is defined by computing the values of $\psi$ after a projection of $\s_0$ onto $\s$.}}\label{fig psi0}
\end{figure}
Figure \ref{fig psi0}. Of course one can entail $3\rho> \mu_0+C_3\,\e_0$ by requiring $\e_0$ small enough with respect to $\mu_0$: in this way, property (c) follows from \eqref{friday3} and \eqref{friday4}. Summarizing, we have shown that if $\mu_0$ is small enough depending on $L$ (that is, depending on $M=M(L)$), and if $\e_0$ is small enough with respect to $\mu_0$ and $|\T|$, then properties (a)--(d) hold with $\rho=\mu_0+C_2\,\e_0$. Up to further decrease the values of $\mu_0$ and $\e_0$ we may entail $\rho\le\mu_1^2$, and thus, thanks to \cite[Theorem 2.6, Proposition B.2]{CiLeMaIC1}, find a $C^{1,1}$-diffeomorphism $f_0$ between $\s_0$ and $\g$ such that $f_0(\bd(\s_0))=\bd(\g)$ and
\[
\|f_0\|_{C^{1,1}(\s_0)}\le C\,,\qquad\|f_0-\Id\|_{C^1(\s_0)}\le C\,\mu_0\,,
\]
\[
\|(f_0-\Id)\cdot\tau_0\|_{C^{1,1}(\s_0)}\le C\,\sup_{\bd(\s_0)}|f_0-\Id|\,,
\]
where $C$ depends on $L$ only. By repeating this construction on every edge $\s_0$ of $\pa\H_0$ we complete the proof of \eqref{friday00} and \eqref{friday01}.

\medskip

\noindent {\it Step four}: With a little abuse of notation, let us denote by $\{\s_i\}_{i=1}^{3N}$ the family of segments such that $\pa\H_0=\bigcup_{i=1}^{3N}\s_i$. For every $i$ let $\tau_i$ denote a constant tangent unit vector to $\s_i$. If we set $g=f_0-\Id$, then we have
\begin{eqnarray*}
  P(\E)-P(\H)=\sum_{i=1}^{3N}\int_{\s_i}\big(|\nabla^{\s_i}g[\tau_i]+\tau_i|-1\big)\,d\H^1\,,
\end{eqnarray*}
where, by $\|g\|_{C^1(\pa\H_0)}\le\mu_0$, $\sqrt{1+t}\ge 1+t/2-t^2/8-C\,|t|^3$ ($t\ge -1$), and provided $\mu_0$ is small enough,
\begin{eqnarray*}
  |\nabla^{\s_i}g[\tau_i]+\tau_i|-1
  &=&\sqrt{1+2 \tau_i\cdot\nabla^{\s_i}g[\tau_i]+|\nabla^{\s_i}g[\tau_i]|^2}-1
  \\
  &\ge& \tau_i\cdot\nabla^{\s_i}g[\tau_i]+\frac{|\nabla^{\s_i}g[\tau_i]|^2}2-\frac{|2\,\tau_i\cdot\nabla^{\s_i}g[\tau_i]|^2}8
  -C\,\mu_0\,|\nabla^{\s_i}g[\tau_i]|^2\,.
\end{eqnarray*}
Let $\S(\H_0)=\{p_j\}_{j=1}^{2N}$, and for $p_j\in\bd(\s_i)$ denote by $v_j^i$ the tangent unit vector to $\s_i$ at $p_j$ pointing outside $\s_i$. In this way,
\[
\sum_{i=1}^{3N}\int_{\s_i}\,\tau_i\cdot\nabla^{\s_i}g[\tau_i]\,d\H^1=\sum_{j=1}^{2N}\sum_{\{i:p_j\in\bd(\s_i)\}}\,v_j^i\,g(p_j)=0\,,
\]
since $\{i:p_j\in\bd(\s_i)\}=\{i_1,i_2,i_3\}$ with $v_j^{i_2}$ and $v_j^{i_3}$ obtained from $v_j^{i_1}$ by counterclockwise rotations of $2\pi/3$ and $4\pi/3$ respectively. Hence, if we set $\nu_i=\tau_i^\perp$, then
  \begin{equation}\label{deficit 1}
    P(\E)-P(\H)\ge \sum_{i=1}^{3N}\int_{\s_i}\frac{|\nu_i\cdot\nabla^{\s_i}g[\tau_i]|^2}2\,d\H^1-C\,\mu_0\,\int_{\s_i}|\nabla^{\s_i}g[\tau_i]|^2\,d\H^1\,.
  \end{equation}
  By \eqref{friday01} and \eqref{deficit 0} we find that
  \[
  \sup_{1\le i\le 3N}\|\tau_i\cdot\nabla^{\s_i}g[\tau_i]\|_{C^0(\s_i)}\le C\,\sqrt\de\,,
  \]
  where $C$ depends on $L$ and $|\T|$. By combining this last inequality with \eqref{deficit 1}, and provided $\mu_0$ is small enough with respect to $L$ and $|\T|$, we find
  \begin{equation}
    \label{combina}
      C\,\sqrt{\de}\ge \sum_{i=1}^{3N}\int_{\s_i}|\nabla^{\s_i}g[\tau_i]|\ge\sum_{i=1}^{3N}\|g-g(p_{j(i)})\|_{C^0(\s_i)}\,,
  \end{equation}
  where for each $i=1,...,3N$ we have picked $p_{j(i)}\in \bd(\s_i)$. By \eqref{deficit 0} we have $|g(p_{j(i)})|\le C\sqrt\de$, so that \eqref{combina} implies
  \begin{equation}
    \label{combina C0}
      C\,\sqrt{\de}\ge\sum_{i=1}^{3N}\|g\|_{C^0(\s_i)}=\|f_0-\Id\|_{C^0(\pa\H_0)}\,.
  \end{equation}
  Since $f_0$ is a bijection between $\pa\H_0$ and $\pa\E$, we find that $\|f_0-\Id\|_{C^0(\pa\H_0)}\ge\hd(\pa\H_0,\pa\E)$ and thus prove \eqref{stabilita hd}. We now notice that if $u:(a,b)\to\R$ is a Lipschitz function, then
  \begin{equation}
    \label{interpol}
    \|u\|_{C^0(a,b)}^2\le 8\,\max\Big\{\Lip(u),\frac1{b-a}\Big\}\,\|u\|_{L^1(a,b)}\,.
  \end{equation}
  Indeed, let $x_0\in(a,b)$ be such that $u(x_0)=\|u\|_{C^0(a,b)}$ and set $L=\Lip(u)$, $r=|u(x_0)|/4L$. If $(x_0,x_0+r)\subset(a,b)$ or $(x_0-r,x_0)\subset(a,b)$, then by integrating $|u(y)|\ge |u(x_0)|-L|x_0-y|$ in $y$ over $(x_0,x_0+r)$ or over $(x_0-r,x_0)$ respectively, we find
  \[
  \int_{(a,b)}|u|\ge r\,|u(x_0)|-L\frac{r^2}2\ge \frac{|u(x_0)|^2}{8L}\,;
  \]
  otherwise one has $b-a\le 2r$ and thus $|u(y)|\ge |u(x_0)|/2$ for every $y\in(a,b)$. In order to complete the proof of \eqref{stabilita f C1} we just need to use \eqref{combina C0} and to combine the first inequality in \eqref{combina} with $\|f_0\|_{C^{1,1}(\pa\H)}\le C$ and with \eqref{interpol} (applied  to the components of $\nabla^{\pa^*\H_0}(f_0-\Id)$).
\end{proof}

\section{Proof of Theorem \ref{thm main periodic}, Theorem \ref{thm pertub volumes} and Theorem \ref{thm pertub metric}}\label{section fine} We start by introducing the following fundamental tool in the study of isoperimetric problems with multiple volume constraints. This kind of construction is originally found in \cite{Almgren76}, and it is fully detailed in our setting in \cite[Sections 29.5-29.6]{maggiBOOK}, see also \cite[Theorem C.1]{CiLeMaIC1}. Since the version of this lemma needed here does not seem to appear elsewhere, we give some details of the proof.

\begin{lemma}[Volume-fixing variations]\label{lemma volume fixing}
  If $\E_0$ is a $N$-tiling of $\T$, $\g\in(0,1]$ and $L>0$, then there exist positive constants $r_0$, $\s_0$, $\e_0$, and $C_0$ (depending on $\E_0$, $L$ and $\g$ only) with the following property: if $\eta\in\R^N$ with $\sum_{h=1}^N\eta_h=0$, $\Phi\in\Lip(\T\times S^1;(0,\infty))$, $\psi\in C^{1,\g}(\T;(0,\infty))$, $x\in\T$, and $\E$ and $\F$ are $N$-tilings of $\T$ with
  \begin{gather}\label{volumefix hp 0}
    \|\Phi\|_{C^{0,1}(\T\times S^1)}+\|\psi\|_{C^{1,\g}(\T)}\le L\,,
    \\
    \label{volumefix hp 1}
    \d(\E,\E_0)\le \e_0\,,
    \\
    \label{volumefix hp 2}
    \F\Delta\E\cc B_{x,r_0}\,,\qquad |\eta|<\s_0\,,
  \end{gather}
  then there exists a $N$-cluster $\F'$ such that
  \begin{eqnarray}\label{volumefix thesis 1}
    \F'\Delta\F&\cc& \T\setminus\ov{B}_{x,r_0}\,,
    \\\label{volumefix thesis 2}
    \int_{\F'(h)}\psi&=&\eta_h+\int_{\E(h)}\psi\,,
    \\\label{volumefix thesis 3}
    |\PHI(\F')-\PHI(\F)|&\le& C_0\,P(\E)\,\Big(\sum_{h=1}^N\Big|\int_{\F(h)}\psi-\int_{\E(h)}\psi\Big|+|\eta|\Big)\,,
    \\\label{volumefix thesis 4}
    |\d(\F',\E)- \d(\F,\E)|&\le&C_0\,P(\E)\,\Big(\sum_{h=1}^N\Big|\int_{\F(h)}\psi-\int_{\E(h)}\psi\Big|+|\eta|\Big)\,.
  \end{eqnarray}
\end{lemma}

\begin{remark}
  {\rm In practice we are going to apply this lemma either with $\eta=0$ and $\F\Delta\E\ne\emptyset$, or with $\eta\ne 0$ and $\F=\E$. In the first case, we are given a compactly supported variation $\F$ of $\E$, and we want to modify $\F$ outside of $B_{x,r_0}$ into a new $N$-tiling $\F'$ so that $\int_{\F'(h)}\psi=\int_{\E(h)}\psi$ for every $h=1,...,N$. In the second case we want to modify $\E$ so that  $\int_{\E(h)}\psi$ is changed into $\eta_h+\int_{\E(h)}\psi$ for every $h=1,...,N$. In both cases, we want to control the change in $\PHI$-energy and the change in distance from $\E$ needed to pass from $\F$ to $\F'$. The name attached to the lemma is motivated by the fact that one usually takes $\psi\equiv1$.}
\end{remark}

\begin{proof}[Proof of Lemma \ref{lemma volume fixing}]
  The basic step consists in picking up a ball $B_{z,\e}$ and notice that if $T\in C^\infty_c(B_{z,\e};\R^2)$ and $f_t(x)=x+t\,T(x)$ for $x\in\T$, then for every Borel set $E\subset\T$ the function $\PSI_E(t)=\int_{f_t(E)}\psi=\int_E\psi(f_t)Jf_t$ is of class $C^{1,\g}(-t_0,t_0)$ with
  \begin{equation}
    \label{derivata psi}
    \|\PSI_E\|_{C^{1,\g}(-t_0,t_0)}\le C\,,\qquad  \Big|\int_{f_t(E)}\psi-\int_E\psi-t\,\int_E\,\Div(\psi\,T)\Big|\le C\,|t|^{1+\g}\,,
  \end{equation}
  where $t_0$ and $C$ denote positive constants depending only on $\g$, $L$, $|\T|$, and $\|T\|_{C^1(\T)}$. Next, one considers two families of balls $\{B_{z_i,\e}\}_{i=1}^M$  and $\{B_{y_i,\e}\}_{i=1}^M$ with $z_i\,,y_i\in\pa^*\E_0(h(i))\cap \pa^*\E_0(k(i))$ (for $1\le h(i)\ne k(i)\le N$ to be properly chosen -- see condition \eqref{rank} below) and with $|z_i-z_j|>2\e$ and $|y_i-y_j|>2\e$ for $1\le i<j\le M$ and $|y_i-z_j|>2\e$ for $1\le i\le j\le M$. For each $i$ we can find $T_i\in C^\infty_c(B_{z_i,\e};\R^2)$ such that
  \begin{gather}\label{fix 1}
  \int_{\E_0(h(i))}\Div(\psi\,T_i)=1=-\int_{\E_0(k(i))}\Div(\psi\,T_i)\,,
  \\\label{fix 2}
  \int_{\E_0(j)}\Div(\psi\,T_i)=0\,,\qquad j\ne h(i),k(i)\,.
  \end{gather}
  Let us consider the smooth map $f:(-t_0,t_0)^M\times\T\to\T$ defined by $f(\tt,x)=x+\sum_{i=1}^M\,t_i\,T_i(x)$, $\tt=(t_1,...,t_M)$, so that for $t_0>0$ small enough $f(\tt,\cdot)$ is a smooth diffeomorphism of $\T$ with
  \begin{equation}
    \label{spt 1}
      \spt(f(\tt,\cdot)-\Id)\cc \bigcup_{i=1}^MB_{z_i,\e}\,.
  \end{equation}
  If we let $\a=(\a_1,...,\a_N)\in C^{1,\g}((-t_0,t_0)^M;\R^N)$ be defined by
  \[
  \a_h(\tt)=\int_{f(\tt,\E(h))}\psi-\int_{\E(h)}\psi\,,\qquad h=1,...,N\,,
  \]
  then $\a((-t_0,t_0)^M)\subset V=\{\eta\in\R^N:\sum_{h=1}^N\eta_h=0\}$, $\|\a\|_{C^{1,\g}((-t_0,t_0)^M)}\le C$, and, by
  \eqref{volumefix hp 0}, \eqref{volumefix hp 1}, \eqref{derivata psi}, \eqref{fix 1} and \eqref{fix 2}, one finds
  \begin{eqnarray}\label{vicino}
  \Big|\frac{\pa\a_{h(i)}}{\pa t_i}(\tt)-1\Big|+\Big|\frac{\pa\a_{k(i)}}{\pa t_i}(\tt)+1\Big|+\max_{j\ne h(i),k(i)}\Big|\frac{\pa\a_{j}}{\pa t_i}(\tt)\Big|\le C\,\e_0\,,
  \end{eqnarray}
  where, from now on, $C$ denotes a constant depending only on $L$, $\g$, $|\T|$, and $\E_0$ (through $\|T_i\|_{C^1(\T)}$). Provided $h(i)$ and $k(i)$ are suitable defined (see \cite[Step one, Proof of Theorem 29.14]{maggiBOOK}) one can entail from \eqref{vicino} that
  \begin{equation}
    \label{rank}
    {\rm dim}\nabla\a(\00)=N-1\,.
  \end{equation}
  By the implicit function theorem there exists $\s_1>0$ and an open neighborhood $U$ of $\00\in\R^M$ such that $\a^{-1}\in C^{1,\g}(V_{\s_1};U)$ with $V_{\s_1}=\{\eta\in V:|\eta|<\s_1\}$, and
  \begin{equation}
    \label{sotto}
    |\a^{-1}(\eta)|\le C\,|\eta|\,,\qquad \forall \eta\in V_{\s_1}\,.
  \end{equation}
  Similarly, we may construct functions $g$ and $\beta$, analogous to $f$ and $\a$, starting from the family of balls $\{B_{y_i,\e}\}_{i=1}^M$. Now let $\F$ be as in \eqref{volumefix hp 2}, and assume that
  \begin{equation}
    \label{s0 r0}
    \s_0+\|\psi\|_{C^0(\T)}\pi\,r_0^2<\s_1\,.
  \end{equation}
  Up to further decrease the value of $r_0$ with respect to $\e$, we may also assume that $\ov{B}_{x,r_0}\cap \ov{B}_{z_i,\e}=\emptyset$ for every $i=1,...,M$, or that $\ov{B}_{x,r_0}\cap \ov{B}_{y_i,\e}=\emptyset$ for every $i=1,...,M$. Without loss of generality we may assume to be in the former case, and set
  \[
  \F'(h)=\big(\F(h)\cap B_{x,r_0}\big)\cup\big(f(\a^{-1}(w),\E(h))\setminus B_{x,r_0}\big)\,,\qquad 1\le h\le N\,,
  \]
  where $w_h$ is defined by the identity
  \[
  \int_{\F(h)\cap B_{x,r_0}}\psi=\eta_h-w_h-\int_{\E(h)\cap B_{x,r_0}}\psi\,,\qquad 1\le h\le N\,.
  \]
  By construction one has \eqref{volumefix thesis 1}. Moreover, by definition of $w_h$, by \eqref{spt 1} and since $\ov{B}_{x,r_0}\cap \ov{B}_{z_i,\e}=\emptyset$ for every $i=1,...,M$, one has
  \begin{eqnarray*}
    \int_{\F'(h)}\psi-\int_{\E(h)}\psi&=&\int_{\F(h)\cap B_{x,r_0}}\psi+\int_{f(\a^{-1}(w),\E(h))\setminus B_{x,r_0}}\psi-\int_{\E(h)}\psi
    \\
    &=&\eta_h-w_h+\int_{f(\a^{-1}(w),\E(h))\setminus B_{x,r_0}}\psi-\int_{\E(h)\setminus B_{x,r_0}}\psi
    \\
    &=&\eta_h-w_h+\int_{f(\a^{-1}(w),\E(h))}\psi-\int_{\E(h)}\psi=\eta_h-w_h+\a_h(\a^{-1}(w))\,.
  \end{eqnarray*}
  By  \eqref{volumefix hp 2} and \eqref{s0 r0} one has $|w|<\s_1$, so that \eqref{volumefix thesis 2} is proved. We now notice that by \cite[Equation (2.9)]{dephilippismaggiARMA}
  \[
  \PHI(f(\tt,E))=\int_{f(\tt,\pa^*E)}\Phi(y,\nu_{f_t(E)}(y))\,d\H^1(y)
  =
  \int_{\pa^*E}\Phi\Big(f_t(x),{\rm cof}\nabla f_t(x)[\nu_E(x)]\Big)\,d\H^1(x)\,,
  \]
  so that, by \eqref{volumefix hp 0}, $|\PHI(f(\tt,E))-\PHI(E)|\le C\,|t|\,P(E)$. By \eqref{sotto} we immediately deduce \eqref{volumefix thesis 3}. Finally \eqref{volumefix thesis 4} is obtained by exploiting \cite[Lemma C.2]{CiLeMaIC1}.
\end{proof}

We now translate the improved convergence theorem for planar bubble clusters from \cite{CiLeMaIC1} in the case of tilings of $\T$. One says that a $N$-tiling $\E$ of $\T$ is {\it $(\Lambda,r_0)$-minimizing} if
\[
P(\E)\le P(\F)+\Lambda\,\d(\E,\F)\,,
\]
whenever $\F$ is a $N$-tiling of $\T$ and  $\E\Delta\F=\bigcup_{h=1}^N\E(h)\Delta\F(h)\cc B_{x,r_0}$ for some $x\in\T$. If $\E$ is a $(\Lambda,r_0)$-minimizing tiling of $\T$, then (by a trivial adaptation of, say, \cite[Theorem 3.16]{CiLeMaIC1}) $\E$ is of class $C^{1,1}$. Moreover, the curves $\g_i$ and the points $p_j$ in \eqref{class C1} are such that each $\g_i$ has distributional curvature bounded by $\Lambda$, and for every $p_j$ there exists exactly three curves from $\{\g_i\}_{i\in I}$ which share $p_j$ as a common boundary point, and meet at $p_j$ by forming three 120 degrees angles.

We notice that, by \eqref{hexagonal honeycomb thm torus}, the reference honeycomb $\H$ is a $(0,\infty)$-minimizing unit-area tiling of $\T$. The following result is what we call an {\it improved convergence theorem}.

\begin{theorem}\label{thm improved convergence}
  Given $\Lambda\ge0$, there exist positive constants $L$ and $\mu_*>0$ (depending on $\Lambda$ and $\H$) with the following property. If $N=|\T|$, $\mu<\mu_*$ and $\{\E_k\}_{k\in\N}$ is a sequence of  $(\Lambda,r_0)$-minimizing $N$-tilings of $\T$ (for some $r_0>0$) with $\d(\E_k,\H)\to 0$ as $k\to\infty$, then there exist $k(\mu)\in\N$ and, for every $k\ge k(\mu)$, a $C^{1,1}$-diffeomorphism $f_k$ with
  \begin{equation}
    \label{ic C11 e C1}
      \sup_{k\ge k(\mu)}\|f_k\|_{C^{1,1}(\pa\H)}\le L\,,\qquad \lim_{k\to\infty}\|f_k-\Id\|_{C^1(\pa\H)}=0\,,
  \end{equation}
  \begin{equation}
    \label{ic tangenziale}
      \ttau_{\H}(f_k-\Id)=0\quad\mbox{on $[\pa\H]_\mu$}\,,\qquad \|\ttau_{\H}(f_k-\Id)\|_{C^1(\pa^*\H)}\le\frac{L}\mu\,\sup_{\S(\H)}|f_k-\Id|\,.
  \end{equation}
  In particular, $\E_k$ is a $(\e_k,\mu,L)$-perturbation of $\H$ whenever $k\ge k(\mu)$.
\end{theorem}

\begin{proof}
  This is a simple variant of \cite[Theorem 1.5]{CiLeMaIC1}, and therefore we omit the details.
\end{proof}


Let us now set
\begin{equation}
  \label{kappa R}
\k=\k(\T)=\inf\,\liminf_{k\to\infty}\frac{P(\F_k)-P(\H)}{\a(\F_k)^2}\,,
\end{equation}
where the infimum is taken among all sequences $\{\F_k\}_{k\in\N}$ of unit-area tilings of $\T$ such that $\a(\F_k)>0$ for every $k\in\N$ and $\a(\F_k)\to0$ as $k\to\infty$. By a compactness argument, Theorem \ref{thm main periodic} is equivalent in saying that $\k>0$.

\begin{lemma}\label{thm selection}
  If $\k=0$, then there exists a sequence of $(\Lambda,r_0)$-minimizing unit-area tilings $\{\E_k\}_{k\in\N}$ such that $\a(\E_k)>0$ for every $k\in\N$, $\a(\E_k)\to 0$ as $k\to\infty$, and
  \begin{equation}\label{behavior di Ek}
  P(\E_k)=P(\H)+o(\a(\E_k)^2)\,,\qquad\mbox{as $k\to\infty$}\,.
  \end{equation}
\end{lemma}

\begin{proof}
  By definition of $\k$, and since we are assuming $\k=0$, there exist unit-area tilings $\{\F_k\}_{k\in\N}$ of $\T$ such that $\a(\F_k)>0$ for every $k\in\N$, and
  \begin{equation}
    \label{Fk}
    \a(\F_k)\to 0\,,\qquad P(\F_k)=P(\H)+o(\a(\F_k)^2)\,,\qquad\mbox{as $k\to\infty$}\,.
  \end{equation}
  For every $k\in\N$, let $\E_k$ be a minimizer in the variational problem
  \[
  \inf\,\Big\{P(\E)+\d(\E,\F_k)^2 \ | \ \text{$\E$ unit-area tiling of $\T$ with $\a(\E)>0$}\Big\}\,.
  \]
  By comparing $\E_k$ with $\F_k$ and then subtracting $P(\H)$ one has
  \begin{equation}
    \label{comparison Fk}
      P(\E_k)-P(\H)+\d(\E_k,\F_k)^2\leq P(\F_k)-P(\H)=o(\a(\F_k)^2)\,.
  \end{equation}
  Since $|\a(\E_k)-\a(\F_k)|\le\d(\E_k,\F_k)$ and $P(\E_k)\ge P(\H)$, we conclude that
  \begin{equation}
    \label{take}
      \lim_{k\to\infty}\frac{\a(\E_k)}{\a(\F_k)}=1\,,
  \end{equation}
  so that, in particular, $\a(\E_k)\to 0$ as $k\to\infty$. Dividing by $\a(\E_k)^2$ in \eqref{comparison Fk} and using \eqref{take}, we complete the proof of \eqref{behavior di Ek}. We now show that each $\E_k$ is $(\Lambda,r_0)$-minimizing in $\T$. Indeed, let $r_0$, $\e_0$, $\s_0$ and $C_0$ be the constants associated by Lemma \ref{lemma volume fixing} to $\E_0=\H$, $\PHI=P$ and $\psi\equiv 1$. Since $\a(\E_k)\to0$, up to translations we have $\d(\E_k,\H)\le\e_0$ for $k$ large. We apply Lemma \ref{lemma volume fixing} with $\E=\E_k$, $\F$ a $N$-tiling with $\E_k\Delta\F\cc B_{x,r_0}$ for some $x\in\T$, and $\eta=0$, to find a unit-area tiling $\F'$ such that
  \begin{eqnarray*}
    &&P(\E_k)\le P(\E_k)+\d(\E_k,\F_k)^2\le P(\F')+\d(\F',\E_k)^2
    \\
    &\le&P(\F)+C_0\,P(\E_k)\,|\vol(\F)-\vol(\E_k)|+\Big(\d(\F,\E_k)+C_0\,P(\E_k)\,|\vol(\F)-\vol(\E_k)|\Big)^2\,.
  \end{eqnarray*}
  Hence $P(\E_k)\le P(\F)+\Lambda\,\d(\E_k,\F)$ thanks to $|\vol(\F)-\vol(\E_k)|\le\d(\F,\E_k)$ and since, for $k$ large enough, $P(\E_k)\le 2\,P(\H)$.
  \end{proof}

\begin{proof}
  [Proof of Theorem \ref{thm main periodic}] We argue by contradiction. If the theorem is false, then $\k=0$ and thus by Lemma \ref{thm selection} there exists a sequence $\{\E_k\}_{k\in\N}$ of $(\Lambda,r_0)$-minimizing unit-area tilings of $\T$ such that $\a(\E_k)>0$, $\a(\E_k)\to 0$ as $k\to\infty$ and
  \[
  P(\E_k)=P(\H)+o(\a(\E_k)^2)\,,\qquad\mbox{as $k\to\infty$}\,.
  \]
  Up to translation we may assume that $\a(\E_k)=\d(\E_k,\H)\to 0$ as $k\to\infty$. Let $L$ and $\mu_*$ be the constants of Theorem \ref{thm improved convergence} (which depends on $\Lambda$ and $\H$) so that for every $\mu<\mu_*$ there exists $k(\mu)\in\N$ such that $\E_k$ is a $(\e_k,\mu,L)$-perturbation of $\H$ for every $k\ge k(\mu)$, with $\e_k\to 0$ as $k\to\infty$. Let $\e_0$ and $\mu_0$ be determined as in Theorem \ref{hexagonal honeycomb thm torus quantitative small} depending on $L$ and $|\T|$. If we set $\mu=\min\{\mu_*,\mu_0\}$ and increase $k(\mu)$ so that $\e_k\le\e_0$ for $k\ge k(\mu)$, then by Theorem \ref{hexagonal honeycomb thm torus quantitative small}, one finds $v_k\in\R^2$ with $|v_k|\le C\,\e_k$ such that
  \[
  P(\E_k)-P(\H)\ge c_0\,\hd(\pa\E_k,v_k+\pa\H)^2\ge c\,\d(\E_k,v_k+\H)^2\ge c\,\a(\E_k)^2\,,
  \]
  for some positive constant $c$. We have thus reached a contradiction, and proved the theorem.
\end{proof}

\begin{proof}[Proof of Theorem \ref{thm pertub volumes}]
   Let $\E_j=\E_{m^j}$ be minimizers in \eqref{variational problem volumes} for a sequence $\{m^j\}_{j\in\N}$ such that $\sum_{h=1}^Nm^j_h=N$, $m_h^j>0$ and $m_h^j\to 1$ as $j\to\infty$. By an explicit construction, for every $j$ large enough we can construct a small deformation $\H_j$ of $\H$ such that $|\H_j(h)|=m_h^j$ and $P(\H_j)\le P(\H)+C\,\max_{1\le h\le N}|m_h^j-1|$, with $C$ independent from $j$. (Alternatively, one can apply Lemma \ref{lemma volume fixing} with $\E_0=\E=\F=\H$, $\PHI=P$, $\psi\equiv1$ and $\eta_h=m_h^j-1$.) As a consequence, $\sup_{j\in\N} P(\E_j)<\infty$, and thus, up to extracting subsequences, $\d(\E_j,\E_0)\to0$ where $\E_0$ is a unit-area tiling of $\T$. In particular,
  \[
  P(\H)\le P(\E_0)\le\liminf_{j\to\infty}P(\E_j)\le\liminf_{j\to\infty}P(\H)+C\,\max_{1\le h\le N}|m_h^j-1|=P(\H)\,.
  \]
  By Hales' theorem, up to a relabeling of $\E_0$, $\E_0=v+\H$ for $v=(t\sqrt{3}\ell,s\ell)$ and $t,s\in[0,1]$. By performing the same relabeling on each $\E_j$, we have $\d(\E_j,v+\H)\to 0$. By exploiting Lemma \ref{lemma volume fixing} as in the proof of Lemma \ref{thm selection} one sees that each $\E_j$ is a $(\Lambda,r_0)$-minimizing tiling in $\T$, and then by arguing as in the proof of Theorem \ref{thm main periodic} we find a constant $L$ (depending on $\Lambda$ and $\H$) such that $\E_j-v$ is an $(\e_j,\mu_0,L)$-perturbation of $\H$ for $\mu_0$ as in Theorem \ref{hexagonal honeycomb thm torus quantitative small} and for $\e_j\to 0$ as $j\to\infty$. By Theorem \ref{hexagonal honeycomb thm torus quantitative small}, for $j$ large enough there exist $v_j\to 0$ and $C^{1,1}$-diffeomorphism $f_j$ between $v_j+\pa\H$ and $\pa\E_j-v$, with
  \begin{eqnarray*}
     C\,\max_{1\le h\le N}|m_h^j-1|\ge P(\E_j)-P(\H)
     \ge c\,\Big(\|f_j-\Id\|_{C^0(v_j+\pa\H)}^2+\|f_j-\Id\|_{C^1(v_j+\pa\H)}^4\Big)\,.
  \end{eqnarray*}
  Theorem \ref{thm pertub volumes} is then deduced by a contradiction argument.
\end{proof}

\begin{proof}[Proof of Theorem \ref{thm pertub metric}]
   In the following we denote by $\E_\de$ a minimizer in \eqref{finsler}, and set
   \[
   \de=\de(\Phi,\psi)=\|\Phi-1\|_{C^0(\T\times S^1)}+\|\psi-1\|_{C^0(\T)}\,,
   \]
   so that $\de<\de_0$. We notice that for every $E\subset\T$ of finite perimeter one has
   \begin{eqnarray}
    \label{bella}
      \Big|\int_E\psi -|E|\Big|&\le& C\,|E|\,\|\psi-1\|_{C^0(\T)}\,,
      \\
    \label{bella 2}
      |\PHI(E)-P(E)|&\le&C\,\min\{P(E)\,,\PHI(E)\}\,\|\Phi-1\|_{C^0(\T\times S^1)}\,,
   \end{eqnarray}
   where in \eqref{bella 2} we have also used the fact that $P(E)\le2\,\PHI(E)$ provided $\de_0\le 1$.

   \medskip

   \noindent {\it Step one}: We claim that, provided $\de_0$ is small enough, then
  \begin{eqnarray}\label{uno}
    \PHI(\E_\de)&\le&2\,P(\H)\,,
    \\\label{due}
    P(\E_\de)&\le&P(\H)+C\,\de\,.
  \end{eqnarray}
  Indeed, by considering an explicit small modification of $\H$ (or by applying Lemma \ref{lemma volume fixing} with $\E=\E_0=\F=\H$ and $\eta\ne 0$) we can construct a $N$-tiling $\H'$ of $\T$ such that $\int_{\H'(h)}\psi=N^{-1}\,\int_\T\psi$ for every $h=1,...,N$ and $\PHI(\H')\le \PHI(\H)+C\,\de$. By $\PHI(\E_\de)\le \PHI(\H')$ and by \eqref{bella 2}
  \begin{equation}
    \label{c}
      \PHI(\E_\de)\le \PHI(\H)+C\,\de\le P(\H)+C\,\de\,,
  \end{equation}
  which implies \eqref{uno}. Again by \eqref{bella 2}, $P(\E_\de)\le\PHI(\E_\de)+C\,\de$, and \eqref{c} gives \eqref{due}.

  \medskip

  \noindent {\it Step two}: We now show that if $\de_j=\de(\Phi_j,\psi_j)\to 0$ and $\E_j$ is a minimizer in \eqref{finsler} associated to $\Phi_j$ and $\psi_j$, then (and up to subsequences and to relabeling the chambers of $\E_j$) $\d(\E_j,v+\H)\to 0$ for some $v=(t\sqrt{3}\ell,s\ell)$, $s,t\in[0,1]$. By \eqref{uno} and since $\PHI_j(E)\ge P(E)/2$ for every $E\subset\T$ we find that $\sup_{j\in\N}P(\E_j)\le 4\,P(\H)$. By compactness, there exists a $N$-tiling $\E_*$ of $\T$ such that $\d(\E_j,\E_*)\to 0$ (up to subsequences). By \eqref{bella}, $\int_{\E_j(h)}\psi_j=N^{-1}\int_{\T}\psi_j$ implies $m_j(h)=|\E_j(h)|\to 1$ for every $h=1,...,N$. In particular, $\E_*$ is a unit-area tiling of $\T$, and thus by \eqref{hexagonal honeycomb thm torus}, by lower semicontinuity and by \eqref{due}
  \begin{equation}
    \label{ciao}
      P(\H)\le P(\E_*)\le\liminf_{j\to\infty}P(\E_j)\le P(\H)\,.
  \end{equation}
  By Hales' theorem, up a relabeling, $\E_*=v+\H$.

  \medskip

  \noindent {\it Step three}: Let $\e_0$, $r_0$, $\s_0$ and $C_0$ be the constants associated to $\E_0=\H$, $\Phi$ and $\psi$ by Lemma \ref{lemma volume fixing}. (Notice that the same constants will work on any translation of $\H$, and that these constants ultimately depend on $L$ and $\g$ only.) By step two we can assume that $\de_0$ is small enough to entail $\d(\E_\de,v_\de+\H)\le\e_0$ for some translation $v_\de$. We now claim that there exist positive constants $r_1\,,c_0>0$ such that
  \begin{equation}
    \label{lb}
    |\E_\de(h)\cap B_{x,r}|\ge c_0\,r^2\,,\qquad\forall x\in\pa\E_\de(h)\,,r<r_1\,,h=1,...,N\,.
  \end{equation}
  This is a classical argument, see for example \cite[Lemma 30.2]{maggiBOOK}, and we include some details just for the sake of completeness. Without loss of generality let us set $h=1$ and fix $x\in\pa\E_\de(1)$ and $r<r_1\le r_0$ such that $P(\E_\de;\pa B_{x,r})=0$. There exists $j\in\{1,...,N\}$ such that
  \begin{equation}
    \label{questa}
      \H^1(\pa^*\E_\de(1)\cap\pa^*\E_\de(j)\cap B_{x,r})\ge \H^1(\pa^*\E_\de(1)\cap\pa^*\E_\de(h)\cap B_{x,r})\,,\qquad\forall h\ne 1,j\,.
  \end{equation}
  If we set $\F(1)=\E_\de(1)\setminus B_{x,r}$, $\F(j)=\E_\de(j)\cup(\E_\de(1)\cap B_{x,r})$ and $\F(h)=\E_\de(h)$ for $h\ne 1,j$, then by applying Lemma \ref{lemma volume fixing} with $\E_0=v_\de+\H$, $\E=\E_\de$, and $\eta=0$ and setting $u(r)=|\E_\de(1)\cap B_{x,r}|$, we find that, if $\e<r_0-r$, then
  \begin{eqnarray*}
    \PHI(\E_\de;B_{x,r+\e})&\le&\PHI(\F;B_{x,r+\e})+C_0\,P(\E_\de) \Big|\int_{\E_\de(1)\cap B_{x,r}}\psi\Big|
    \\
    &\le&
    \PHI(\E_\de;B_{x,r+\e})+\hat{\PHI}(B_{x,r};\E_\de(1))
    \\
    &&-\int_{\pa^*\E_\de(1)\cap\pa^*\E_\de(j)\cap B_{x,r}}\hat{\Phi}(y,\nu_{\E_\de(1)}(y))\,d\H^1+C\,u(r)\,,
  \end{eqnarray*}
  where we have set $\hat\Phi(x,\nu)=(\Phi(x,\nu)+\Phi(x,-\nu))/2$. In particular, by \eqref{questa} and by $2\ge\Phi\ge 1/2$, for every $h\ne 1$ one finds
  \[
  \H^1(\pa^*\E_\de(1)\cap\pa^*\E_\de(h)\cap B_{x,r})\le C\big(\H^1(\E_\de(1)\cap\pa B_{x,r})+u(r)\big)\,,
  \]
  i.e.
  \[
  P(\E_\de(1);B_{x,r})\le C(u'(r)+u(r))\,,\qquad\mbox{for a.e. $r<r_1$}\,.
  \]
  By adding $u'(r)=\H^1(\E_\de(1)\cap\pa B_{x,r})$ to both sides we find that
  \[
  C(u'(r)+u(r))\ge P(\E_\de(1)\cap B_{x,r})\ge 2\sqrt{\pi\,u(r)}\,.
  \]
  In particular if $r_1$ is small enough to give $C\,u(r)\le C\sqrt{\pi r_1^2\,u(r)}\le \sqrt{\pi\,u(r)}$, then we find $\sqrt{u(r)}\le C\,u'(r)$ for a.e. $r<r_1$. This proves \eqref{lb}.

  \medskip

  \noindent {\it Step four}: We now conclude the proof. Again by step two and by Lemma \ref{lemma volume fixing}, one can find a unit-area tiling $\E_\de'$ of $\T$ such that $P(\E_\de')\le P(\E_\de)+C\,\de$ and $\d(\E_\de',\E_\de)\le C\,\de$. By Theorem \ref{thm main periodic}  and up to permutations of the chambers of $\E_\de$, we find a translation $v_\de$ such that
  \[
  c\,\d(\E_\de',v_\de+\H)^2\le P(\E_\de')-P(\H)\le P(\E_\de)-P(\H)+C\,\de\le C\,\de\,,
  \]
  where in the last inequality we have used \eqref{due}. Since $\d(\E_\de',v_\de+\H)\ge \d(\E_\de,v_\de+\H)-\d(\E_\de',\E_\de)$ we conclude
  \[
  \d(\E_\de,v_\de+\H)^2\le C\,\de\,.
  \]
  Setting for the sake of brevity $v_\de=0$, we now pick $x\in\pa\E_\de(1)$ such that $\dist(x,\pa\H(1))\ge\dist(y,\pa\H(1))$ for every $y\in\pa\E_\de(1)$. Let $r=\min\{r_1,\dist(x,\pa\H(1))\}$, so that either $B_{x,r}\subset\T\setminus\H(1)$ or $B_{x,r}\subset\H(1)$. In particular, provided $\de_0$ is small enough with respect to $c_0$, either
  \[
  \d(\E_\de,\H)\ge|\E_\de(1)\setminus\H(1)|\ge |\E_\de(1)\cap B_{x,r}|\ge c_0\,r^2\ge c_0\,\dist(x,\pa\H(1))^2\,,
  \]
  or
  \begin{eqnarray*}
  \d(\E_\de,\H)&\ge&|\H(1)\setminus\E_\de(1)|\ge |B_{x,r}\setminus\E_\de(1)|=\Big|\bigcup_{h=2}^NB_{x,r}\cap\E_\de(h)\Big|
  \\
  &\ge& (N-1)c_0\,r^2\ge c_0\,\dist(x,\pa\H(1))^2\,;
  \end{eqnarray*}
  in both cases, $\pa\E_\de(1)\subset I_\e(\pa\H(1))$ for $\e=C\,\sqrt{\d(\E_\de,\H)}$. By the same argument (based on area density estimates for $\H$, which hold trivially) one finds that $\pa\H(1)\subset I_\e(\pa\E_\de(1))$.
\end{proof}

\bibliography{references}
\bibliographystyle{is-alpha}
\end{document}

%% file: hexagon.pstex_t
\begin{picture}(0,0)%
\includegraphics{hexagon.eps}%
\end{picture}%
\setlength{\unitlength}{3947sp}%
\begingroup\makeatletter\ifx\SetFigFont\undefined%
\gdef\SetFigFont#1#2#3#4#5{%
  \reset@font\fontsize{#1}{#2pt}%
  \fontfamily{#3}\fontseries{#4}\fontshape{#5}%
  \selectfont}%
\fi\endgroup%
\begin{picture}(5683,2361)(225,-1755)
\put(5893,-409){\makebox(0,0)[lb]{\smash{{\SetFigFont{10}{12.0}{\rmdefault}{\mddefault}{\updefault}{\color[rgb]{0,0,0}$\frac32\ell\alpha$}%
}}}}
\put(4348,-1692){\makebox(0,0)[lb]{\smash{{\SetFigFont{10}{12.0}{\rmdefault}{\mddefault}{\updefault}{\color[rgb]{0,0,0}$\sqrt{3}\ell\beta$}%
}}}}
\put(3774,-1194){\makebox(0,0)[lb]{\smash{{\SetFigFont{11}{13.2}{\rmdefault}{\mddefault}{\updefault}{\color[rgb]{0,0,0}$H$}%
}}}}
\put(3272,128){\makebox(0,0)[lb]{\smash{{\SetFigFont{11}{13.2}{\rmdefault}{\mddefault}{\updefault}{\color[rgb]{0,0,0}$\T$}%
}}}}
\put(2468,-732){\makebox(0,0)[lb]{\smash{{\SetFigFont{9}{10.8}{\rmdefault}{\mddefault}{\updefault}{\color[rgb]{0,0,0}$x_1$}%
}}}}
\put(1173,-1485){\makebox(0,0)[lb]{\smash{{\SetFigFont{10}{12.0}{\rmdefault}{\mddefault}{\updefault}{\color[rgb]{0,0,0}$\sqrt{3}\ell$}%
}}}}
\put(1295,-567){\makebox(0,0)[lb]{\smash{{\SetFigFont{11}{13.2}{\rmdefault}{\mddefault}{\updefault}{\color[rgb]{0,0,0}$\hat H$}%
}}}}
\put(240,-534){\makebox(0,0)[lb]{\smash{{\SetFigFont{10}{12.0}{\rmdefault}{\mddefault}{\updefault}{\color[rgb]{0,0,0}$\ell$}%
}}}}
\put(1004,469){\makebox(0,0)[lb]{\smash{{\SetFigFont{9}{10.8}{\rmdefault}{\mddefault}{\updefault}{\color[rgb]{0,0,0}$x_2$}%
}}}}
\put(654,-1001){\makebox(0,0)[lb]{\smash{{\SetFigFont{9}{10.8}{\rmdefault}{\mddefault}{\updefault}{\color[rgb]{0,0,0}$(0,0)$}%
}}}}
\put(2421,-167){\makebox(0,0)[lb]{\smash{{\SetFigFont{10}{12.0}{\rmdefault}{\mddefault}{\updefault}{\color[rgb]{0,0,0}$2\ell$}%
}}}}
\end{picture}%

%% file: pi.pstex_t
\begin{picture}(0,0)%
\includegraphics{pi.eps}%
\end{picture}%
\setlength{\unitlength}{3947sp}%
\begingroup\makeatletter\ifx\SetFigFont\undefined%
\gdef\SetFigFont#1#2#3#4#5{%
  \reset@font\fontsize{#1}{#2pt}%
  \fontfamily{#3}\fontseries{#4}\fontshape{#5}%
  \selectfont}%
\fi\endgroup%
\begin{picture}(1589,2011)(812,-1604)
\put(1003,  0){\makebox(0,0)[lb]{\smash{{\SetFigFont{9}{10.8}{\rmdefault}{\mddefault}{\updefault}{\color[rgb]{0,0,0}$E$}%
}}}}
\put(1698,270){\makebox(0,0)[lb]{\smash{{\SetFigFont{9}{10.8}{\rmdefault}{\mddefault}{\updefault}{\color[rgb]{0,0,0}$a_i$}%
}}}}
\put(2314,-679){\makebox(0,0)[lb]{\smash{{\SetFigFont{9}{10.8}{\rmdefault}{\mddefault}{\updefault}{\color[rgb]{0,0,0}$\Pi$}%
}}}}
\put(1519,-511){\makebox(0,0)[lb]{\smash{{\SetFigFont{9}{10.8}{\rmdefault}{\mddefault}{\updefault}{\color[rgb]{0,0,0}$\ell_i$}%
}}}}
\put(1953,-21){\makebox(0,0)[lb]{\smash{{\SetFigFont{9}{10.8}{\rmdefault}{\mddefault}{\updefault}{\color[rgb]{0,0,0}$f(\s_i)$}%
}}}}
\put(1757,-220){\makebox(0,0)[lb]{\smash{{\SetFigFont{9}{10.8}{\rmdefault}{\mddefault}{\updefault}{\color[rgb]{0,0,0}$\s_i$}%
}}}}
\end{picture}%

%% file: psi0.pstex_t
\begin{picture}(0,0)%
\includegraphics{psi0.eps}%
\end{picture}%
\setlength{\unitlength}{3947sp}%
\begingroup\makeatletter\ifx\SetFigFont\undefined%
\gdef\SetFigFont#1#2#3#4#5{%
  \reset@font\fontsize{#1}{#2pt}%
  \fontfamily{#3}\fontseries{#4}\fontshape{#5}%
  \selectfont}%
\fi\endgroup%
\begin{picture}(4543,1646)(932,-1771)
\put(4767,-508){\makebox(0,0)[lb]{\smash{{\SetFigFont{10}{12.0}{\rmdefault}{\mddefault}{\updefault}{\color[rgb]{0,0,0}$\nu_0$}%
}}}}
\put(3388,-1278){\makebox(0,0)[lb]{\smash{{\SetFigFont{10}{12.0}{\rmdefault}{\mddefault}{\updefault}{\color[rgb]{0,0,0}$\s$}%
}}}}
\put(3587,-1002){\makebox(0,0)[lb]{\smash{{\SetFigFont{10}{12.0}{\rmdefault}{\mddefault}{\updefault}{\color[rgb]{0,0,0}$\s_0$}%
}}}}
\put(4011,-656){\makebox(0,0)[lb]{\smash{{\SetFigFont{10}{12.0}{\rmdefault}{\mddefault}{\updefault}{\color[rgb]{0,0,0}$\g$}%
}}}}
\put(1416,-1579){\makebox(0,0)[lb]{\smash{{\SetFigFont{10}{12.0}{\rmdefault}{\mddefault}{\updefault}{\color[rgb]{0,0,0}$\mu_0$}%
}}}}
\put(2590,-271){\makebox(0,0)[lb]{\smash{{\SetFigFont{10}{12.0}{\rmdefault}{\mddefault}{\updefault}{\color[rgb]{0,0,0}$3\rho\ge\mu_0+C\,\e_0$}%
}}}}
\put(2932,-1708){\makebox(0,0)[lb]{\smash{{\SetFigFont{10}{12.0}{\rmdefault}{\mddefault}{\updefault}{\color[rgb]{0,0,0}$\Id+\psi\,\nu_0$}%
}}}}
\put(5460,-1214){\makebox(0,0)[lb]{\smash{{\SetFigFont{10}{12.0}{\rmdefault}{\mddefault}{\updefault}{\color[rgb]{0,0,0}$\tau_0$}%
}}}}
\put(5017,-1002){\makebox(0,0)[lb]{\smash{{\SetFigFont{10}{12.0}{\rmdefault}{\mddefault}{\updefault}{\color[rgb]{0,0,0}$v$}%
}}}}
\end{picture}%